\documentclass[a4paper, 10pt ]{article}
\usepackage[a4paper,left=3cm,right=3cm,top=2.5cm,bottom=2.5cm]{geometry}
\usepackage{hyperref}
\hypersetup{colorlinks,linkcolor={blue},citecolor={blue},urlcolor={black}}
\usepackage{subfig}
\usepackage{pgfplots}
\usepackage{pgfplotstable}
\usepackage{mathtools}
\usepackage{multicol}
\usepackage{comment}
\usepackage{booktabs}
\pgfplotsset{compat=1.5}
\usepackage{amssymb}
\usepackage{url}
\usepackage{bm}
\usepackage{pdfsync}
\usepackage{float}
\usepackage{tabularx}
\usepackage{enumerate}
\usepackage{array}
\usepackage{xspace}
\usepackage{tikz}
\usepackage{tikz-cd}
\usepackage{tikzsymbols}
\usetikzlibrary{calc,trees,positioning,arrows,chains,shapes.geometric,%
    decorations.pathreplacing,decorations.pathmorphing,shapes,%
    matrix,shapes.symbols, decorations.markings, patterns,fit}

\usepackage{siunitx}

\usepackage{authblk}

\usepackage[draft,inline,marginclue]{fixme}
\usepackage{mathrsfs}
\usepackage{color, colortbl}
\usepackage{multirow}

\usepackage{amsthm}
\usepackage{amsmath}
\usepackage{stmaryrd}
\usepackage[ruled,vlined,linesnumbered]{algorithm2e}
\usepackage{graphicx}
\usepackage{booktabs}
\usepackage{cleveref}

\usepackage[utf8]{inputenc}
\usepackage{tgpagella}
\usepackage{bm}
\usepackage{amsmath, amsfonts,amssymb}
\usepackage{mathtools}
\newtheorem{assumption}{Assumption}

\DeclareMathOperator{\dist}{dist}

\let\epsilon\varepsilon

\let\phi\varphi

\let\theta\vartheta

\newcommand{\spacenameinnorm}{\mathcal{D}}
\newcommand{\spacenameinnormhom}{\mathcal{G}}
\newcommand{\cG}{\mathcal{G}}

\newcommand{\wnormdg}[5]{\left\vert\kern-0.25ex\left\vert\kern-0.25ex\left\vert #1 
      \right\vert\kern-0.25ex\right\vert\kern-0.25ex\right\vert_{\spacenameinnorm^{#3}_{#4}(#5)} }
\newcommand{\wnormdghom}[5]{\left\vert\kern-0.25ex\left\vert\kern-0.25ex\left\vert #1 
      \right\vert\kern-0.25ex\right\vert\kern-0.25ex\right\vert_{\spacenameinnormhom^{#3}_{#4}(#5)} }

\newcommand{\fc}{{\mathfrak{c}}}
\newcommand{\fC}{{\mathfrak{C}}}

\newcommand{\dalpha}{\partial^\alpha}

\newcommand{\alpham}{{|\alpha|}}

\newcommand{\ugamma}{{\underline{\gamma}}}

\newcommand{\ubeta}{{\underline{\beta}}}

\newcommand{\bM}{\mathbf{M}}

\newcommand{\cK}{\mathcal{K}}

\newcommand{\bu}{\bm{u}}
\newcommand{\bv}{\bm{v}}
\newcommand{\bx}{\bm{x}}
\newcommand{\bW}{\bm{W}}
\newcommand{\bn}{\bm{n}}
\newcommand{\bt}{\bm{t}}

\newcommand{\bw}{\bm{w}}

\newcommand{\bbf}{\bm{f}}

\newcommand{\bzero}{{\bm{0}}}

\newcommand{\cX}{\mathcal{X}}

\newcommand{\N}{\mathbb{N}}
\newcommand{\range}[1]{\{1, \dots, #1\}}

\newcommand{\Ccoer}{C_{\mathrm{coer}}}
\newcommand{\Ccont}{C_{\mathrm{cont}}}
\newcommand{\Ccoerh}{C_{{\mathrm{coer}}, h}}
\newcommand{\Cconth}{C_{{\mathrm{cont}}, h}}

\graphicspath{{./figures/}}

\usepackage{tcolorbox}
\tcbuselibrary{most}

\newcommand{\bbP}{\Omega}

\FXRegisterAuthor{gs}{ags}{\color{red}GS}
\FXRegisterAuthor{fr}{afr}{\color{blue}FR}
\FXRegisterAuthor{gr}{agr}{\color{red}GR}
\FXRegisterAuthor{cs}{acs}{\color{blue}S}
\FXRegisterAuthor{fp}{afp}{\color{green}FP}

\newtheorem{theorem}{Theorem}

\theoremstyle{remark}

\newcommand\eqsvd{\mathrel{\stackrel{\makebox[0pt]{\mbox{\normalfont\tiny SVD}}}{=}}}

\newcommand{\x}{\ensuremath{\mathbf{x}}}
\newcommand{\ub}{\ensuremath{\mathbf{u}}}

\newcommand{\R}{\ensuremath{\mathbb{R}}}
\newcolumntype{C}[1]{>{\centering\arraybackslash}m{#1}}

\definecolor{Gray}{gray}{0.9}

\begin{document}

% \title{SET TITLE} 
% \title{Numerical validation of exponential KnW decay of the INS in polygonal domains}
% \title{Exponential bound of the Kolmogorov $n$-width for reduced order model of viscous, incompressible flows in polygons}
 \title{ROM for Viscous, Incompressible Flow in Polygons --\\
        exponential $n$-width bounds and convergence rate} 

\author[]{Francesco~Romor\footnote{francesco.romor@wias-berlin.de, Weierstrass Institute for Applied Analysis and Stochastics (WIAS), Berlin, Germany},
  Federico~Pichi\footnote{federico.pichi@sissa.it, Mathematics Area, mathLab, SISSA, Trieste,
  Italy},
  Giovanni~Stabile\footnote{giovanni.stabile@santannapisa.it, Sant'Anna School of Advanced Studies, The Biorobotics Institute, Pisa, Italy
  },
  Gianluigi~Rozza\footnote{gianluigi.rozza@sissa.it, Mathematics Area, mathLab, SISSA, Trieste,
    Italy},
  Christoph~Schwab\footnote{christoph.schwab@sam.math.ethz.ch, Seminar for Applied Mathematics, ETH Z\"{u}rich, Z\"{u}rich, Switzerland}}

% \author[]{}
% \author[]{}
% \author[]{}
% \author[]{}

\maketitle

\begin{abstract}
We demonstrate exponential convergence of 
Reduced Order Model (ROM) approximations for mixed boundary value problems of the 
stationary, incompressible Navier-Stokes equations in plane, polygonal domains $\Omega$.
Admissible boundary conditions comprise mixed BCs, no-slip, slip and 
open boundary conditions, subject to corner-weighted analytic boundary data and volume forcing. 
The small data hypothesis is assumed to ensure existence of a unique weak solution in the sense of Leray-Hopf. 
Recent results on corner-weighted, analytic regularity of velocity and pressure fields in $\Omega$, imply exponential convergence rates of so-called mixed $hp$-Finite Element Methods in $H^1(\Omega)^2\times L^2(\Omega)$ on sequences of geometric partitions of $\Omega$, with corner-refinement. Based on these exponential convergence rate bounds, we infer exponential bounds for the Kolmogorov $n$-widths of solution sets for analytic forcing and boundary data. This implies corresponding exponential convergence rates of POD Galerkin methods that are based on truth solutions which are obtained offline from low-order, divergence stable mixed Finite Element discretizations.
Numerical experiments confirm the exponential rates and the theoretical results.
\end{abstract}

% \tableofcontents
% \listoffixmes

\section{Introduction}
\label{sec:intro}
Model Order Reduction (MOR) strategies have achieved widespread success in mathematics and engineering applications, providing an efficient framework for studying parametrized systems in real-time and many-query contexts \cite{BennerSystemDataDrivenMethods2021,BennerSnapshotBasedMethodsAlgorithms2020}. In particular, significant efforts have been devoted to analyzing incompressible fluid dynamics problems governed by parametrized Partial Differential Equations (PDEs), which are of great interest for both their mathematical implications and industrial applications. The full-order numerical approximation of these problems—that is, the high-fidelity discretization— often entails an unfeasible computational burden.

A common reduction technique exploits the Proper Orthogonal Decomposition (POD), a compression strategy that captures the dominant modes of variation of a problem, to perform a Galerkin projection onto the low-dimensional linear space spanned by these modes.
Greedy methodologies have been developed to address systems with high-dimensional parametric spaces, where an efficient and reliable sampling of the parametric manifold is required.
Both techniques use a linear approximation of the field via a projection onto a low-dimensional space; however, additional assumptions on the parametric dependence are typically needed to enable fast model evaluations.
More specifically, once the parametric PDE is sampled during the offline phase by computing snapshots at different points in the parametric space, affinely dependent forms can be pre-computed so that the online phase, designed for rapid evaluations, is independent of the full-order simulation degrees of freedom (for more details, we refer to~\cite{Hesthaven2016,Quarteroni2016}). Furthermore, such linear reduction techniques assume the parametric solutions set to be well approximated by a low-dimensional linear space, which is not always the case in practice, i.e.\ for moving discontinuities. Indeed, these issues are related to the slow-decay of the so-called \textit{Kolmogorov n-width} of the solution set~\cite{PinkusBook}, representing a measure of how well it can be approximated by low-dimensional linear spaces.

A first MOR approach for incompressible flows has been proposed in~\cite{SirovichTurbulenceDynamicsCoherent1987a}, where a POD based technique was introduced for the computation of coherent structures in the framework of turbulence. In~\cite{Veroy2005}, a Greedy approach has been developed, in combination with a posteriori error estimates, for rapid and rigorous approximation of parametrized Navier-Stokes problems. Such complex scenarios stimulated the discussion towards several important directions: (i) the inf-sup stability of the MOR approach by augmenting the velocity space via supremizers~\cite{RozzaStabilityReducedBasis2007,StabRozzaPODG2018}, (ii) the recovery of the efficiency thanks to hyper-reduction or tensor-assembly approaches when nonlinear terms prevent fast online computations~\cite{BarraultEmpiricalInterpolationMethod2004,ChaturantabutNonlinearModelReduction2010}, and (iii) the lack of uniqueness for the solution in bifurcating context~\cite{HerreroRBReducedBasis2013,DengLoworderModelSuccessive2020a,PichiDeflationbasedCertifiedGreedy2025}.
Recently, in order to properly address more difficult phenomena closer to engineering applications, e.g.\ large 3D fluid dynamics, different Reynolds numbers, turbulent regimes, and complex geometries (see e.g.\ \cite{BruntonMachineLearningFluid2020,VinuesaEnhancingComputationalFluid2022,QuainiBridgingLargeEddy2024,StabileEfficientGeometricalParametrization2020}) researchers in the MOR field are increasingly adopting Scientific Machine Learning (SciML) approaches to develop linear/nonlinear and intrusive/non-intrusive methodologies exploiting data-driven knowledge for fast evaluations, overcoming the classical issues of the previously discussed linear reduction strategies, i.e.\ the affine-decomposition assumption and the slow decay of Kolmogorov $n$-widths.

It is well recognized that many recent SciML methods still use POD decompositions to construct linear ansatz spaces that are then coupled with neural networks~\cite{Fresca2022PODDLROM,PichiArtificialNeuralNetwork2021a,Romor2023Explicable,Cohen2023Nonlinear}. This stands in contrast to truly nonlinear ansatz spaces provided by autoencoders~\cite{LeeModelReductionDynamical2020,MaulikReducedorderModelingAdvectiondominated2021,RomorNonlinearManifoldReducedOrder2023,PichiGraphConvolutionalAutoencoder2024}. Although such methods can be more efficient and easier to implement in a non-intrusive manner, they ultimately inherit the approximation properties of POD basis functions and can be analyzed in terms of the decay of the Kolmogorov $n$-widths of solution sets. Upper bounding these widths for solution sets of planar viscous, incompressible flow in polygons is one purpose of the present paper.

% These approaches are commonly designed with the help of Machine Learning architectures 
% (see,e.g., \cite{LeeModelReductionDynamical2020,MaulikReducedorderModelingAdvectiondominated2021,RomorNonlinearManifoldReducedOrder2023,PichiGraphConvolutionalAutoencoder2024}).
% They share the advantage of exploiting data to build the reduced model, 
% but ML based approaches often lack of a strong mathematical background, 
% resulting in less robust and trustworthy models and limiting their potentialities. 
% Despite this, recently rigorous mathematical descriptions have been developed
% of different ML learning tasks 
% \cite{FrancoPracticalExistenceTheorem2024,MarcatiExponentialConvergenceDeep2023,deryckErrorEstimatesPhysics2022,MarcatiExpressionRatesNeural2024}, which could cover ML-ROM type algorithms in the future.

%%%%%%%%%%%%%%%%%%%%%%%%%%%%%%%%%%%%%%%%%%%%%%%%%%%%%%%%%%%%%5
\subsection{Previous Results}
\label{sec:PrevRes}
%%%%%%%%%%%%%%%%%%%%%%%%%%%%%%%%%%%%%%%%%%%%%%%%%%%%%%%%%%%%%%
% For well-posed linear, elliptic BVPs, 
% it is well-known that the corresponding
% operators induce isomorphisms between the 
% solution space and the data space.

% Written in abstract form, 
% with elliptic differential operator $A$
% and given data $F$, 
% %
% $$ A U = F $$
% %
% In the case that
% $A\in \cL_{is}(\cX,\cY)$ 
% with 
% $\cX$ and $\cY$ suitable separable Hilbert spaces,

% \cite{HMS21_971,DivConfRBM2019,StabRozzaPODG2018}
% \cs{[Attn: please discuss in detail the infsup stability of the ROMs used for the numerical experiments]}

For a generic parametric PDE, we associate to each parameter in a Banach space $\mu\in\mathcal{P}$ its corresponding solution $L(\mu)\in\cX$, through the parameter-to-solution map $L:\mathcal{P}\rightarrow(\cX, \lVert\cdot\rVert_{\cX})$, where $(\cX, \lVert\cdot\rVert_{\cX})$ is an appropriately chosen Banach space depending on the PDE. The Kolmogorov $n$-width (KnW) of the image set $L(\mathcal{P})\subset\cX$ is defined as 
\begin{equation}
     \begin{split}
     d_n(L(\mathcal{P}), \cX)&=\inf_{\substack{W\subset\cX \\\text{dim}(W)=n}}\sup_{\mu\in \mathcal{P}}\inf_{w\in W} \lVert L(\mu)-w\rVert_{\cX}\\
     &=\inf_{\substack{W\subset\cX \\\text{dim}(W)=n}}\text{dist}(W, L(\mathcal{P}))_{L^{\infty}(\mathcal{P};\cX)}.
     \end{split}
\end{equation}

The efficiency of ROMs in terms of accuracy over reduced dimension $n$ depends on the KnW decay. Some results show algebraic convergence for infinite affine  parametric dependencies~\cite{cohen2011analytic}, 
algebraic and exponential (for finite parameters) convergence for holomorphic solution maps~\cite{cohen2015approximation}. 
Estimates of $n$-widths of classical approximation spaces were given, among others, in~\cite{PinkusBook}.
While for \emph{linear, elliptic PDEs} the isomorphism property of the differential operator implies  KnW bounds for the solution sets in terms of corresponding bounds for the data (e.g.\ \cite{Melenk2000}), for \emph{nonlinear, elliptic PDEs} this needs to be verified in a case-by-case fashion. 
For non-elliptic, linear PDE such as transport-dominated PDEs,
slow KnW decays with the order $n^{-1/2}$ are known,
e.g.\ for linear advection and wave equations with discontinuous initial conditions~\cite{Ohlberger2016Reduced, Arbes2023Kolmogorov}.
%%%%%%%%%%%%%%%%%%%%%%%%%%%%%%%%%%%%%%%%%%%%%%%%%%%%%%%%%%%%%5
\subsection{Contributions}
\label{sec:Contr}
%%%%%%%%%%%%%%%%%%%%%%%%%%%%%%%%%%%%%%%%%%%%%%%%%%%%%%%%%%%%%5
We prove exponential smallness of Kolmogorov $n$-widths of solution
sets of the stationary Navier-Stokes equations (NSE) in a polygonal domain $\bbP$
subject to mixed boundary conditions, including also
open (outflow) boundary.
We admit analytic data (boundary data and
forcing terms) subject to a small data hypothesis, to ensure uniqueness of (Leray-Hopf) solutions.

We exhibit a class of MOR approximation algorithms which 
display, in agreement with the general theory of such methods,
exponential convergence in terms of the number of bases, 
for configurations of engineering interest. We 
precise this empirical observation for the nonlinear, stationary
NSE modelling viscous, incompressible flow
in polygons with a novel, theoretical result.
It is based on recent \emph{weighted, analytic regularity} of velocity
and pressure field established in \cite{HMS21_971}, which imply corresponding exponential $n$-width bounds in $H^1(\Omega)^2 \times L^2(\Omega)$ which we establish here for solution sets of the NSE subject to analytic volume forcing. We verify these mathematical results for two benchmarks of engineering interest: viscous, incompressible flow over a backward-facing step geometry, and viscous, incompressible flow past a cube.

%%%%%%%%%%%%%%%%%%%%%%%%%%%%%%%%%%%%%%%%%%%%%%%%%%%%%%%%%%%%%%%%%5
\subsection{Layout of this note}
\label{sec:Layout}
%%%%%%%%%%%%%%%%%%%%%%%%%%%%%%%%%%%%%%%%%%%%%%%%%%%%%%%%%%%%%%%%%%
In Section~\ref{sec:VIFPolyg}, we set the notation and specify the 
model problem of stationary, viscous and incompressible flow in a polygon $\Omega$.
We also provide the weak formulation and state known results on existence 
of weak solutions in the sense of Leray-Hopf, 
and uniqueness, subject to the usual small data hypothesis.

Section~\ref{sec:AnReg} reviews analytic regularity of weak solutions, 
subject to an analyticity assumption on the data, in scales of corner-weighted
Sobolev spaces of Kondrat'ev type.

Section~\ref{sec:ExpNWid} then presents the main result, an exponential bound on
the Kolmogorov $n$-widths of solutions, for analytic data as introduced in 
Section~\ref{sec:AnReg}.

Section~\ref{sec:results} presents numerical results, and supplementary 
proofs of the ``high-fidelity'' mixed Finite-Element discretizations which were 
used to generate the ROM subspaces in the reported experiments.

Section~\ref{sec:conclusions} has a summary of the main findings of this note,
and indicates some directions for further work.

\section{Viscous and Incompressible Flow in Polygons}
\label{sec:VIFPolyg}
%
%%%%%%%%%%%%%%%%%%%%%%%%%%%%%%%%%%%%%%%%%%%%%%%%%%%%%%%%%%%%%%%%%%%%%%%%%
\subsection{Setting and Notation}
\label{sec:SetNot}
%%%%%%%%%%%%%%%%%%%%%%%%%%%%%%%%%%%%%%%%%%%%%%%%%%%%%%%%%%%%%%%%%%%%%%%%%
We adopt the setting and notation of \cite{HMS21_971}:
$\bbP$ denotes
a polygon with $n\geq 3$ straight, open sides $\Gamma_i$ 
and $n$ corners $\fC = \{\fc_1, \dots, \fc_n\}$ 
with interior opening angles $\omega_i\in (0,2\pi)$, $i=1,2,...,n$ 
(enumerated in counterclockwise order, and modulo $n$, 
i.e.\ we identify $\Gamma_n$ with $\Gamma_0$ and $\Gamma_{n+1}$ with $\Gamma_1$, etc.),
so that $\fc_i = \overline{\Gamma_i} \cap \overline{\Gamma_{i+1}}$.
Let $\Gamma_D$, $\Gamma_N$, and $\Gamma_G$ be a disjoint partition of the
boundary $\Gamma = \partial \bbP$ of $\bbP$
comprising each of $n_D\geq 1$, $n_N\geq 0$ and $n_G\geq 0$
many sides of $\bbP$, respectively,
with $n=n_D+n_N+n_G$.
We denote by $\bn:\Gamma \to \R^2$
the exterior unit normal vector to $\bbP$,
defined almost everywhere on $\Gamma$, which belongs to  $L^\infty(\Gamma;\R^2)$,
and by
$\bt \in L^\infty(\Gamma; \R^2)$ correspondingly the unit tangent vector to $\Gamma$,
pointing in counterclockwise tangential direction.
%%%%%%%%%%%%%%%%%%%%%%%%%%%%%%%%%%%%%%%%%%%%%%%%%%%%%%%%%%%%%%%%%%%%%%%%%
\subsection{Statement of the boundary value problem}
\label{sec:BVP}
%%%%%%%%%%%%%%%%%%%%%%%%%%%%%%%%%%%%%%%%%%%%%%%%%%%%%%%%%%%%%%%%%%%%%%%%%
In $\bbP$, we consider the
stationary, incompressible Navier-Stokes equations, which read
\begin{equation}
  \label{eq:NS}
  \begin{alignedat}{2}
    -\nabla\cdot\sigma(\bu, p) + (\bu\cdot\nabla)\bu &= \bbf\quad &&\text{in }\bbP,
    \\
    \nabla\cdot\bu &= 0 \quad&&\text{in }\bbP,
    \\
    \bu & = \bzero \quad &&\text{on } \Gamma_{D},
     \\
    \sigma(\bu, p)\bn & = \bzero \quad &&\text{on } \Gamma_{N},
    \\
    (\sigma(\bu, p)\bn)\cdot \bt  = 0 \text{ and } \bu\cdot\bn&=0 \quad &&\text{on }\Gamma_{G}.
  \end{alignedat}
\end{equation}
We shall refer to the BCs on $\Gamma_{D}$, $\Gamma_N$, and $\Gamma_G$ as
no-slip, open, and slip boundary conditions, respectively.
%
%%%%%%%%%%%%%%%%%%%%%%%%%%%%%%%%%%%%%%%%%%%%%%%%%%%%%%%%%%%%%%%%%%%%%%%%%%%%%%%%%%%%%%%%%%%%%%%%%%%%%%%%%%%%%
\subsection{Variational Formulation}
\label{sec:VarForm}
%%%%%%%%%%%%%%%%%%%%%%%%%%%%%%%%%%%%%%%%%%%%%%%%%%%%%%%%%%%%%%%%%%%%%%%%%%%%%%%%%%%%%%%%%%%%%%%%%%%%%%%%%%%%%
Weak solutions of the NSE \eqref{eq:NS} in the sense of Leray-Hopf
satisfy the NSE \eqref{eq:NS} in variational form.
To state it,
we introduce standard Sobolev spaces in $\bbP$.
\emph{Throughout the remainder of this article, we shall work under}
\begin{assumption}\label{ass:Dirich}
[Domain Assumption] 
\newline
The boundary value problem \eqref{eq:NS} satisfies the following conditions:
\begin{enumerate}
\item \label{item:Lip}
$\bbP$ is a bounded, connected polygon with a finite number $n$ of straight sides,
denoted by $\Gamma_i$, $i=1,...,n$,
and with Lipschitz boundary $\Gamma = \partial \bbP$.
\item \label{item:Dirich}
$n_D\geq 1$.
\end{enumerate}
\end{assumption}
Assumption \ref{ass:Dirich} implies 
that all interior opening angles $\omega_i$ at corners $\fc_i$ of $\bbP$ are in $(0, 2\pi)$
and that there is at least one side of $\bbP$ where homogeneous Dirichlet (``no-slip'') boundary
conditions are applied.
As a consequence, $\Gamma = \Gamma_N \cup \Gamma_G$ is excluded in \cite{HMS21_971} and also
from the present analysis.

Furthermore, 
Item \ref{item:Dirich} ensures that the linearization of the Navier-Stokes equations,
i.e., the Stokes problem, admits unique variational velocity field solutions $\bu$,
possibly with pressure $p$ unique up to constants if $\Gamma=\Gamma_D$.

The space of velocity fields of variational solutions
to the Navier-Stokes equations \eqref{eq:NS}
is denoted as
\begin{equation}
\label{eq:DefW}
\bW = \left\{ \bv \in [H^1(\bbP)]^2: \bv = \bzero\text{ on }\Gamma_D, \, \bv\cdot\bn = 0 \text{ on }\Gamma_G\right\}.
\end{equation}
We denote by $\bW^* $ its dual,
with identification of $L^2(\bbP)^2 \simeq [L^2(\bbP)^2]^*$.
We also define $Q = L^2(\bbP)$ if $|\Gamma_D|<|\Gamma|$ (i.e., if not the entire boundary is
a Dirichlet boundary)
and set $Q = L^2_0(\bbP) := L^2(\bbP)/\R$ in the case that $\Gamma = \Gamma_D$.

We are interested in variational solutions $(\bu,p)$ of \eqref{eq:NS}.
To state the corresponding variational
formulation, we introduce the usual bi- and trilinear forms (e.g. \cite{FMRT2001})
\begin{equation}\label{eq:AOB}
\begin{aligned}
{a}(\bu,\bv) 
&\coloneqq 2\nu\int_\bbP \sum_{i,j=1}^2[\epsilon(\bu)]_{ij} [ \epsilon(\bv) ]_{ij}\, d\bx\;,
     % \frac{\nu}{2} \int_P 
     % (\nabla \bu + (\nabla \bu)^\top): (\nabla \bu + (\nabla \bu)^\top) d\bx\;,
\\
{b}(\bu,p) 
& :=  
-\int_\bbP p\nabla\cdot \bu \, d\bx\;,
\\
{t}(\bw;\bu,\bv) 
& := 
\int_{\bbP}((\bw\cdot\nabla)\bu)\cdot\bv \, d\bx \;.
\end{aligned}
\end{equation}
With these forms, we state the variational formulation of \eqref{eq:NS}:
find $(\bu,p)\in \bW\times Q$
such that
\begin{equation}\label{eq:NSweak}
\begin{aligned}
{a}(\bu,\bv) + {t}(\bu;\bu,\bv) + {b}(\bv,p) 
& = 
\int_\bbP \bbf\cdot \bv \, d\bx \;, \quad &&\forall \bv\in \bW,
\\
{b}(\bu,q) & =  0 \;, \quad &&\forall q\in Q.
\end{aligned} 
\end{equation}
%
% for all $\bv\in \bW$ and all $q\in Q$.
%%%%%%%%%%%%%%%%%%%%%%%%%%%%%%%%%%%%%%%%%%%%%%%%%%%%%%%%%%%%%%%%%%%%%%%%%%%%%%%%%%%%%%%%%%%%%%%%%%%%%%%%%%%%
\subsection{Existence and uniqueness of solutions via the small data hypothesis}
\label{sec:NSE-ex}
%%%%%%%%%%%%%%%%%%%%%%%%%%%%%%%%%%%%%%%%%%%%%%%%%%%%%%%%%%%%%%%%%%%%%%%%%%%%%%%%%%%%%%%%%%%%%%%%%%%%%%%%%%%%%
%

We recapitulate results on existence and uniqueness of variational
solutions of the NSE \eqref{eq:NSweak}.
As is well-known, uniqueness of such solutions in the stationary case
requires a small data hypothesis. To state it,
we introduce the coercivity constant of the viscous (diffusion) term
\begin{equation*}
\Ccoer \coloneqq 
\inf_{\substack{\bv \in \bW\\ \|\bv \|_{H^1(\bbP)} = 1}} 
       2\int_\bbP \sum_{i,j=1}^2[\epsilon(\bv)]_{ij} [ \epsilon(\bv) ]_{ij}\,d\bx,
\end{equation*}
and the continuity constant for the trilinear transport term
\begin{equation*}
\Ccont \coloneqq \sup_{\substack{\bu, \bv, \bw\in \bW 
\\
\|\bu\|_{H^1(\bbP)} =\|\bv\|_{H^1(\bbP)} =\|\bw\|_{H^1(\bbP)} = 1}}
\int_\bbP((\bu \cdot\nabla)\bv)\cdot\bw \, d\bx\; .
\end{equation*}
\begin{assumption}\label{ass:SmData}
[Small Data] 
The forcing term $\bbf$ in \eqref{eq:NS} satisfies the small data assumption
$$
\| \bbf \|_{\bW^*} \leq \frac{\Ccoer^2 \nu^2}{4\Ccont}.
$$
\end{assumption}
The following existence and uniqueness result is then classical, see e.g.\
\cite[Theorem 3.2]{OrltSandig95}.
\begin{theorem}
\label{thm:existence-weak-solution}
{[Existence and Uniqueness of Leray-Hopf solutions for small data]}
  Suppose that Assumption \ref{ass:Dirich} and Assumption~\ref{ass:SmData} hold.

  Then, there exists a solution $(\bu, p)\in \bW \times L^2(\bbP)$ to \eqref{eq:NS} with right-hand side $\bbf$. Moreover, the velocity field $\bu$ is unique in $\bM$, having denoted
  \begin{equation*}
    \bM \coloneqq \left\{ \bv\in\bW: \|\bv\|_{H^1(\bbP)} \leq \frac{\Ccoer \nu}{2\Ccont}\right\}
  \;.
  \end{equation*}
\end{theorem}
As we assumed above $n_D \geq 1$, there is always at least one side of $\bbP$
where homogeneous Dirichlet (``no-slip'') BCs are imposed.

\section{Analytic Regularity}
\label{sec:AnReg}
We discuss the analytic regularity of solutions. 
As the linearization of the NSE \eqref{eq:AOB}, the Stokes equations, 
with the boundary operators in \eqref{eq:NS} 
are
strongly elliptic in the sense of Agmon-Douglis-Nirenberg, 
it is classical that they satisfy local elliptic regularity shifts
scales of Sobolev spaces, which hold up to smooth parts of the boundary $\Gamma$.

It is also known that these regularity results fail in the vicinity of corners,
where corresponding shift theorems require 
\emph{corner-weighted Sobolev spaces due to V.A. Kondrat'ev}, see, e.g.~\cite{Mazya2010} and the references there.

We recapitulate recent results on \emph{analytic regularity shifts} 
from \cite{HMS21_971} in analytic scales of corner-weighted Sobolev spaces.
The characterization of solutions in these corner-weighted Sobolev spaces
underpins the proof of exponential $n$-width bounds 
and thus, in turn, convergence rate estimates of the MOR/ROM.
%%%%%%%%%%%%%%%%%%%%%%%%%%%%%%%%%%%%%%%%%%%%%%%55
\subsection{Corner-weighted Sobolev spaces}
\label{sec:WgtSpc}
%%%%%%%%%%%%%%%%%%%%%%%%%%%%%%%%%%%%%%%%%%%%%%%55
We introduce in a polygon $\bbP$ as defined in Section~\ref{sec:SetNot} analytic,
corner-weighted Sobolev spaces for the velocity field $\bu$ and the pressure field $p$.
As it is well-known from \cite{Mazya2010},
in these scales of spaces, 
for certain ranges of the weight exponents,
for the NSE 
there hold elliptic regularity shifts of any finite order,
and also analytic regularity shifts~\cite{HMS21_971}.
We now state these, as they are required in the proof 
of exponential $n$-width bounds on solutions of the NSE 
\eqref{eq:NSweak}.
%%%%%%%%%%%%%%%%%%%%%%%%%%%%%%%%%%%%%%%%%%%%%%%55
\subsubsection{Finite Order Spaces}
\label{sec:FinOrdSpc}
%%%%%%%%%%%%%%%%%%%%%%%%%%%%%%%%%%%%%%%%%%%%%%%55
For $x\in \bbP$ and $i\in\range{n}$,
let $r_i(x) \coloneqq \dist(x, \fc_i)$.
We define the \emph{corner weight function}
\begin{equation*}
  \Phi_{\ubeta} (x)\coloneqq \prod_{i=1}^nr_i^{\beta_i}(x).
\end{equation*}
In $\bbP$, for $j, k\in \N_0$ and $\ugamma \in\R^n$,
we introduce so-called \emph{non-homogeneous, corner-weighted Sobolev norms}.
They are, for $\ell\in \N_0$, $k\in\N$ with $k> \ell$, and $\ubeta\in \R^n$
given by
\begin{equation}
  \label{eq:Hspace}
\|v \|^2_{H^{k, \ell}_{\ubeta}(\bbP)} 
\coloneqq 
\| v\|_{H^{\ell-1}(\bbP)}^2 
+ 
\sum_{\ell\leq \alpham \leq k} \| \Phi_{\ubeta + \alpham -\ell} \dalpha v\|_{L^2(\bbP)}^2,
\end{equation}
with the convention that the first term is omitted when $\ell=0$.
%
%%%%%%%%%%%%%%%%%%%%%%%%%%%%%%%%%%%%%%%%%%%%%%%55
\subsubsection{Corner-weighted Analytic Classes $B^\ell_{\ubeta}(\bbP)$}
\label{sec:AnRegSpc}
%%%%%%%%%%%%%%%%%%%%%%%%%%%%%%%%%%%%%%%%%%%%%%%55
With the (non-homogeneous) 
corner-weighed spaces $H^{k, \ell}_{\ubeta}(\bbP)$, 
we define corner-weighted analytic classes in the usual way. 
For $\ell = 0,1,2$, 
we denote
\begin{equation}
  \label{eq:Bspace}
\begin{split}
  B^\ell_{\ubeta}(\bbP)\coloneqq \Bigg\{ v\in \bigcap_{k\geq \ell}H^{k,\ell}_{\ubeta}(\bbP):&
     \exists C, A>0 \;\text{such that} \\
     & \| \Phi_{\ubeta + \alpham -\ell} \dalpha v\|_{L^2(\bbP)}
       \leq
     CA^{\alpham-\ell}(\alpham-\ell)!, \, \forall \alpham \geq \ell\Bigg\} 
     \;.
\end{split}
\end{equation}
These corner-weighted, analytic classes were introduced by 
I. Babu\v{s}ka and B.Q. Guo in the 80ies 
(see \cite{Babuska1988Reg,Babuska1988}) 
in order to prove the exponential convergence rate of the so-called 
$hp$-version of the Finite Element Method in polygonal domains $\bbP$. 
Their relevance for the presently considered boundary value problems
of viscous, incompressible flow in polygonal domains is discussed
next.
%%%%%%%%%%%%%%%%%%%%%%%%%%%%%%%%%%%%%%%%%%%%%%%55
\subsection{Analytic Regularity Result}
\label{sec:AnRegRes}
%%%%%%%%%%%%%%%%%%%%%%%%%%%%%%%%%%%%%%%%%%%%%%%%%
With these preparations, we are now in position to state 
the main result \cite[Theorem 2.13]{HMS21_971} 
on analytic regularity of Leray-Hopf solutions of the NSE in the polygon
$\bbP$ subject to analytic data. 
It generalizes earlier results \cite{Guo2006a} 
for the linearized Stokes problem.

\begin{theorem}
\label{th:analytic-u}
{[Analytic Regularity of velocity and pressure in corner-weighted spaces]}
With the notation from Section~\ref{sec:SetNot},
let $\ubeta=(\beta_1,\dots,\beta_n) \in (0,1)^n$ be such that around each
corner $\fc_i$ for $i=1,...,n$, $\beta_i\in(1-\kappa_i,1)\cap(0,1)$ where
$\kappa_i$ is defined as in \cite[Eqn.(2.19)]{HMS21_971}
with respect to corner $\fc_i$, 
in the interval $I=(0,\omega_i)$, 
and to the operator pencil  
$\mathcal{A}_i(\lambda)$
for the linearized (Stokes) boundary value problem  
as defined in \cite[Eqn.~(2.18)]{HMS21_971}.

Suppose that Assumption \ref{ass:Dirich} holds and let $(\bu,p) \in \bW\times Q$ be the weak solution to \eqref{eq:NSweak} with right-hand side $\bbf$.
Suppose also that the data is analytic in the sense that
$\bbf \in [ B^0_{\ubeta}(\bbP) ]^2\cap \bW^{*}$ holds, 
and that $\bbf$ is small, i.e., Assumption \ref{ass:SmData} holds.

Then there holds the analytic regularity shift
\begin{equation*}
(\bu , p) \in 
[B^2_{\ubeta}(\bbP)]^2 \times B^1_{\ubeta}(\bbP).
%subset  [B^2_{\ubeta}(P) ]^2\times B^1_{\ubeta}(P) 
\end{equation*}
\end{theorem}

\section{Exponential Bounds of $n$-widths of solutions}
\label{sec:ExpNWid}

The previous analytic regularity results 
for the velocity and the pressure
imply exponential bounds on the 
Kolmogorov $n$-widths of the Leray solutions, 
in $H^1(\Omega)^2\times L^2(\Omega)$, 
which we will now develop. 
As it is well-known (see e.g.~\cite{CSMSuri99,DSThWhpDGStok})
exponential approximation rates in terms of the dimension
of subspaces (``number of degrees of freedom'')
are afforded in $H^1(\bbP)$ and in $L^2(\bbP)$ 
on bounded subsets of 
$ B^2_{\ubeta}(\bbP) \subset H^1(\bbP)$ 
and of 
$ B^1_{\ubeta}(\bbP) \subset L^2(\bbP)$, 
respectively.
These exponential approximate rate bounds are realized
by the so-called $hp$-version of the Finite Element Method.

The definition of the KnW of the 
compact (in $H^1(\bbP)^2 \times L^2(\bbP)$) solution set
\begin{equation}\label{eq:K}
\cK 
:= 
B^2_{\ubeta}(\bbP)^2 \times B^1_{\ubeta}(\bbP) \subset H^1(\bbP)^2 \times L^2(\bbP),
\end{equation}
then implies corresponding bounds for the 
KnW of $\cK$ in $H^1(\bbP)^2 \times L^2(\bbP)$ 
and thus also for the MOR approximations.
%%%%%%%%%%%%%%%%%%%%%%%%%%%%%%%%%%%%%%%%%%%%%%%%%%%%%%%%%%%%%%%%%%%%%%%%%%%%%%
\begin{theorem}\label{thm:dnK}
{[Exponential KnW bound of solution set for analytic data]}
For the polygon $\bbP$, and for the NSE \eqref{eq:NSweak} under the conditions
Assumption~\ref{ass:Dirich} and the small data assumption
in Theorem~\ref{th:analytic-u}, in particular with data $f\in B^0_\beta(\bbP)^2$,
the corresponding set of solutions belongs to $\cK$ in \eqref{eq:K}.

There exist constants $b,C > 0$ such that for all $n\geq 1$ holds
$$
d_n(\cK, H^1(\bbP)^2 \times L^2(\bbP)) 
\leq 
C\exp(-b n^{1/3}) 
\;.
$$
%%%%%%%%%%%%%%%%%%%%%%%%%%%%%%%%%%%%%%%%%%%%%%%%%%%%%%5
\end{theorem}
\begin{proof}
The proof will be obtained by an upper bound on the KnW 
of the corner-weighted, analytic solution set $\cK$ in \eqref{eq:K}
of the NSE furnished by Theorem~\ref{th:analytic-u}. 
We recall the definition (see, e.g., \cite{PinkusBook})
\begin{align}\label{eq:KnWBd}
d_n(\cK, &H^1(\bbP)^2 \times L^2(\bbP))
=\\
&\inf_{\mathbf{V}_n \times Q_n \subset  \bW\times Q} 
\sup_{(\bu,p) \in \cK}
\inf_{(\bv_n,q_n) \in \mathbf{V}_n\times Q_n} 
\| \bu - \bv_n \|_{H^1(\bbP)^2} 
+
\| p   -  q_n  \|_{L^2(\bbP)} \nonumber
\;.
\end{align}
Here, the first infimum is over all pairs $(\mathbf{V}_n, Q_n)$ 
of subspaces of dimension at most $n$ of 
$\bW$ and of $L^2(\bbP)$, respectively.
The supremum is taken over the set $\cK$ defined in \eqref{eq:K}
which is a compact subset of $\bW\times Q$.

Due to Theorem~\ref{th:analytic-u}, the unique (under the small data
hypothesis in Theorem~\ref{thm:existence-weak-solution}) 
solution $(\bu,p)\in \bM \times Q$ of the NSE boundary value problem \eqref{eq:NSweak} 
belongs to $\cK \cap (\bM\times Q)$.

The second infimum in the definition \eqref{eq:KnWBd} 
of the KnW $d_n(\cK, H^1(\bbP)^2 \times L^2(\bbP))$
is taken over all pairs $(\bv_n,q)$ in subspaces
$\mathbf{V}_n \subset \bW$ and $Q_n\subset L^2(\bbP)$ 
of dimension at most $n$.

We prove \eqref{eq:KnWBd} by 
majorizing $d_n(\cK, H^1(\bbP)^2 \times L^2(\bbP))$
with 
$$
d_n(\cK, H^1(\bbP)^2 \times L^2(\bbP))
\leq 
\sup_{(\bu,p) \in \cK}
\inf_{(\bv_n,q_n) \in \mathbf{V}^{hp}_n \times Q^{hp}_n} 
\| \bu - \bv_n \|_{H^1(\bbP)^2} 
+
\| p   -  q_n  \|_{L^2(\bbP)} 
\;.
$$
Here, 
we will use specific subspaces 
$\mathbf{V}^{hp}_n\subset \bW \subset H^1(\bbP)^2$ 
of continuous, piecewise polynomial velocity
fields 
and 
$Q^{hp}_n\subset L^2(\bbP)$ 
of discontinuous, piecewise polynomial pressure fields
on sequences of corner-refined, geometric partitions 
of $\bbP$ from the so-called $hp$-Finite Element Methods.

With these spaces in hand, we leverage known exponential
approximation rate bounds from \cite{FS20_2675} 
to bound the error in 
velocity approximation,
$\| \bu - \bv_n \|_{H^1(\bbP)^2}$,
and, from \cite{DSThWhpDGStok} for the 
pressure approximation error 
$\| p   -  q_n  \|_{L^2(\bbP)}$.

Specifically, in \cite[Theorem~1]{FS20_2675}, upon noting that 
the corner-weighted classes $\cG^\delta_{\ubeta}$ in \cite{FS20_2675}
coincide, for $\delta = 1$, with $\bu \in  B^2_{\ubeta}(\bbP)^2$ 
in the present paper, 
there holds with $d=2$ the exponential bound 
$$
\inf_{\bv_n \in \mathbf{V}^{hp}_n} \| \bu - \bv_n \|_{H^1(\bbP)^2} 
\leq C\exp( - b n^{1/3}) 
\;,
$$
with some constants $b,C>0$ which are independent of $n$.

We remark that in \cite[Theorem~1]{FS20_2675} only homogeneous
Dirichlet BCs on all of $\partial\bbP$ were considered. 
Inspecting the proof of  \cite[Theorem~1]{FS20_2675} however,
also homogeneous Dirichlet BCs on sides of $\bbP$ for each component
of $\bu, \bv_n$ as arise in \eqref{eq:NS} on $\Gamma_D$ and on $\Gamma_G$ 
are covered.

For the corresponding pressure approximation result, we observe that 
the regularity $p\in B^1_{\ubeta}(\bbP)$ furnished by Theorem~\ref{th:analytic-u}
implies exponential convergence for $hp$-spaces $Q^{hp}_n\subset L^2(\bbP)$, 
according to \cite[Theorem~3]{DSThWhpDGStok}. 

The geometric mesh constructions 
in the proofs of \cite[Theorem~1]{FS20_2675} and \cite[Theorem~3]{DSThWhpDGStok}
differ. 
This is inconsequential here, as the partitions are only used to establish the 
\emph{exponential consistency bounds} 
for approximations of $\bu$ and of $p$ from 
\emph{some} $n$-dimensional subspaces $\mathbf{V}_n\subset \bW$
and $Q_n\subset L^2(\bbP)$.
\end{proof}

\section{Numerical Results}
\label{sec:results}
The aim of our numerical experiments is to validate the exponential bound $\exp(-bn^{1/3})$ in Theorem~\ref{thm:dnK}, showing that there exist specific parametrization spaces $\mathcal{P}$ for the source terms $\bbf\in \mathcal{P}\subset [B^0_{\ubeta}(\bbP) ]^2\cap \bW^{*}$ in~\ref{eq:NS} that attain a KnW decay close to such exponential decay. 

We will consider the test cases \textit{flow over a backward-facing step} and \textit{flow past a cube in a channel} 
with Dirichlet homogeneous boundary conditions in subsection~\ref{subsec:dir}, 
with mixed Dirichlet-Neumann homogeneous boundary conditions in subsection~\ref{subsec:mixed} 
and with mixed  Dirichlet-Neumann homogeneous and slip boundary conditions in subsection~\ref{subsec:mixed_slip}.

As parametrization spaces $\mathcal{P}=\mathcal{P}_N \subset[ B^0_{\ubeta}(\bbP) ]^2\cap \bW^{*}$, we will consider $N$ dimensional subsets of the set of eigenfunctions of the Dirichlet Laplacian, for varying $N\in\mathbb{N}$.
Being the eigenfunctions scalar fields in $\Omega$, 
we set the $x$-component of $\bbf_i\in\mathcal{P}_{N}$ with $i\in[1, N]\subset\mathbb{N}$ 
equal to the $i$-th eigenfunction, while the $y$-component is set to zero. 
The source terms are also properly scaled such that the small data hypothesis is satisfied, i.e.\:
\begin{equation*}
  \begin{split}
  \mathcal{P}_{N}=\Bigg\{\bbf_{i}=(f_i, 0)\in H^1_0(\Omega)^2,\ &i\in\{1,\dots, N\}:\\\ -\Delta f_i = \lambda_i f_i, \;\;& \mbox{in}\;\; \Omega, \; f_i|_{\partial \Omega} = 0,\ \text{and}\ \| \bbf_i\|_{\bW^*}\leq \nu^2 \frac{\Ccoer^2}{4\Ccont}\Bigg\}\;.
  \end{split}
\end{equation*}
We remark that, in order to apply Theorem~\ref{thm:dnK}, we need the validity of the assumption $\mathcal{P}_{N} \subset[ B^0_{\ubeta}(\bbP) ]^2\cap \bW^{*}$ from Theorem~\ref{th:analytic-u}.
This can be proved exploiting that the solution $U(x)$ to the Poisson problem 
\begin{equation*}
  \begin{alignedat}{2}
    -\Delta U &= 1 \quad &&\text{in }\bbP,
    \\
    U & = 0 \quad &&\text{on } \partial\bbP,
  \end{alignedat}
\end{equation*}
belongs to $B^0_{\ubeta}(\bbP)$ \cite{Babuska1988Reg}, and that the $i$-th eigenfunction associated to the eigenvalue $\lambda_i$ of the Dirichlet Laplacian in $\Omega$ can be bounded by $U(x)$:
\[
  |f_i(x)| \leq \lambda_i \lVert f_i\rVert_{\infty} U(x),\quad\forall x\in\Omega,
\]
thanks to a result from~\cite{Moler1968Bounds}, reported in Section 6.2.1 in~\cite{Grebenkov2013Geometrical}.

The chosen parameter space $\mathcal{P}_{N}$ is not commonly used in engineering applications. Typical parameters include the viscosity, the magnitude of the inflow as non-homogeneous Dirichlet boundary condition and parameters that affect the deformation of the computational domain in the context of shape optimization. However, preliminary numerical investigations showed that choosing the viscosity or the inflow magnitude as varying parameters in an interval resulted in a KnW decay faster than $\exp(-bn^{1/3})$. This can be seen in the worked out problem 12 Navier-Stokes system for a backward-facing step in~\cite{Rozza2024} reduced with the POD-Galerkin and the Discrete Empirical Interpolation Methods~\cite{ChaturantabutNonlinearModelReduction2010}.
%%%%%%%%%%%%%%%%%%%%%%%%%%%%%%%%%%%%%%%%%%%%%%%%%%%%%%%%%%%%%%%%%%%%%%%%%%%%%
\subsection{Discretization with divergence-stable, mixed conforming Finite Elements}
\label{subsec:discreteINSFEM}
In the numerical experiments, 
we consider the incompressible Navier-Stokes (INS) equations with homogeneous Dirichlet boundary conditions (subsection~\ref{subsec:dir}), 
homogeneous mixed Dirichlet-Neumann boundary conditions (subsection~\ref{subsec:mixed}),   
or mixed Dirichlet-Neumann homogeneous and slip boundary conditions (subsection~\ref{subsec:mixed_slip}), with $\bbf\in\mathcal{P}_N\subset[ B^0_{\ubeta}(\bbP) ]^2\cap \bW^{*}$:
\begin{equation}
  \label{eq:INS}
  \begin{alignedat}{2}
    -\nabla\cdot\sigma(\bu, p) + (\bu\cdot\nabla)\bu &= \bbf\quad &&\text{in }\bbP,
    \\
    \nabla\cdot\bu &= 0 \quad&&\text{in }\bbP,
    \\
    \bu & = \bzero \quad &&\text{on } \Gamma_{D},
     \\
    \sigma(\bu, p)\bn & = \bzero \quad &&\text{on } \Gamma_{N},
    \\
    (\sigma(\bu, p)\bn)\cdot \bt  = 0 \text{ and } \bu\cdot\bn&=0 \quad &&\text{on }\Gamma_{G},
  \end{alignedat}
\end{equation}
the mixed variational FE discretization is based on the following variational formulation: find $(\bu, p)\in\bW\times Q$ that satisfy $\forall(\bv, q)\in\bW\times Q$
\begin{align*}
  \nu\int_{\Omega}\nabla \bu \cdot \nabla \bv d\x - \int_{\Omega} p\ \nabla\cdot\bv d\x + \int_{\Omega}q\nabla\cdot\bu d\x + \int_{\Omega}\left(\left(\nabla\bu\right)\bu\right)\cdot\bv d\x-\int_{\Omega}\mathbf{f}\cdot\bv d\x = 0,
\end{align*}
with
\begin{equation*}
  \bW = \left\{ \bv \in [H^1(\bbP)]^2: \bv = \bzero\text{ on }\Gamma_D,\quad \bv\cdot\mathbf{n} = \bzero\text{ on }\Gamma_G\right\}, \quad Q=L^2(\Omega).
\end{equation*}
We discretize this system of PDEs with the LBB-stable Lagrangian finite element spaces
$\ub_h\in[\mathbb{P}^2(\mathcal{T}_h)]^2$ and
$p_h\in\mathbb{P}^{1}(\mathcal{T}_h)$, 
where $\mathcal{T}_h$ is a conforming, shape-regular triangulation. We define $\mathcal{P}_{N,h}$ as the discrete counterpart of the parameter space $\mathcal{P}_N$:
\begin{equation}
  \begin{split}
  \label{eq:fhspace}
  \mathcal{P}_{N, h}=\Bigg\{\bbf_{i, h}=(f_{i,h}, 0)\in [\mathbb{P}^2(\mathcal{T}_h)]^2\ :\ &\int_{\Omega}\nabla f_{i,h}\cdot\nabla g_{h} = \lambda_{n, h}\int_{\Omega} f_{i,h}\cdot\nabla g_{h},\\\; &\forall g_{h}\in \mathbb{P}^2(\mathcal{T}_h),\ i\in\{1,\dots, N\}\Bigg\}\;,
  \end{split}
\end{equation}
satisfying the additional small data hypothesis,$\ \| \bbf_{i, h}\|_{\bW^*}\leq \nu^2 \Ccoerh^2/4\Cconth$, with discrete coercivity and continuity constant
\begin{align*}
  \Ccoerh &\coloneqq 
  \inf_{\substack{\bv_h \in [\mathbb{P}^2(\mathcal{T}_h)]^2\\ \|\bv_h \|_{H^1(\bbP)} = 1}} 
          2\int_\bbP \sum_{i,j=1}^2[\epsilon(\bv_h)]_{ij} [ \epsilon(\bv_h) ]_{ij}\,d\bx,\\
  \Cconth &\coloneqq \sup_{\substack{\bu_h, \bv_h, \bw_h\in [\mathbb{P}^2(\mathcal{T}_h)]^2 
  \\
  \|\bu_h\|_{H^1(\bbP)} =\|\bv_h\|_{H^1(\bbP)} =\|\bw_h\|_{H^1(\bbP)} = 1}}
  \int_\bbP((\bu_h \cdot\nabla)\bv_h)\cdot\bw_h \, d\bx\; .
\end{align*}

The discretized weak formulation reads, for $\bbf_h\in\mathcal{P}_{N, h}$: 
find $(\bu_h, p_h)\in[\mathbb{P}_{\Gamma}^2(\mathcal{T}_h)]^2\times\mathbb{P}^1(\mathcal{T}_h)$, 
that satisfy for all $(\bv_h, q_h)\in[\mathbb{P}_{\Gamma}^2(\mathcal{T}_h)]^2\times\mathbb{P}^1(\mathcal{T}_h)$:
\begin{equation}
  \label{eq:weakINS}
  \begin{split}
  \nu\int_{\Omega}\nabla \bu_h \cdot \nabla \bv_h d\x - \int_{\Omega} p_h\ \nabla\cdot\bv_h d\x + &\int_{\Omega}q_h\nabla\cdot\bu_h d\x\\ & + \int_{\Omega}\left(\left(\nabla\bu_h\right)\bu_h\right)\cdot\bv_h d\x-\int_{\Omega}\bbf_h\cdot\bv_h d\x = 0,
\end{split}
\end{equation}
where we have introduced the following notation 
\begin{equation*}
  \mathbb{P}_{\Gamma}^2(\mathcal{T}_h) = \left\{ \bv_h \in [\mathbb{P}^2(\mathcal{T}_h)]^2: \bv_h = \bzero\text{ on }\Gamma_D, \quad \bv_h\cdot\mathbf{n} = \bzero\text{ on }\Gamma_G\right\}.
\end{equation*}
%%%%%%%%%%%%%%%%%%%%%%%%%%%%%%%%%%%%%%%%%%%%%%%%%%%%%%%%%%%%%%%%%%
\subsection{Reduced basis with Proper Orthogonal Decomposition}
\label{subsec: pod}
We denote by $\mathcal{K}_N$ the set of solutions $(\ub, p)$ of the INS with source term $\bbf\in\mathcal{P}_N$:
\begin{equation*}
(\ub, p)\in\mathcal{K}_N \subset\cK 
:= 
B^2_{\ubeta}(\bbP)^2 \times B^1_{\ubeta}(\bbP) \subset H^1(\bbP)^2 \times L^2(\bbP),
\end{equation*}
and its discrete counterpart for $\bbf_{i, h}\in\mathcal{P}_{N, h}$ with $(\bu_h, p_h)\in\mathcal{K}_{N, h}\subset[\mathbb{P}_{\Gamma}^2(\mathcal{T}_h)]^2\times\mathbb{P}^1(\mathcal{T}_h)$.

From the discrete solutions $(\bu_h, p_h)\in\mathcal{K}_{N, h}$, we compute the $n$-dimensional linear subspaces $\mathbf{V}_{\text{POD},n}\times Q_{\text{POD},n}\subset  \bW\times Q$ through the proper orthogonal decomposition~\cite{Hesthaven2016}: fixed $N>n>0$, with $|\mathcal{K}_{N, h}|=|\mathcal{P}_{N, h}|=N$ the cardinality of the finite discrete parameter space $\mathcal{P}_{N, h}$, we collect column-wise the discrete solutions $\{(\bu^i_h, p^i_h)\}_{i=1}^{N}$ in two matrices:
\begin{equation*}
  M_{\bu_h}=(\bu_h^1,\dots,\bu_h^N)\in\mathbb{R}^{d_u\times N},\qquad M_{p_h}=(p_h^1,\dots,p_h^N)\in\mathbb{R}^{d_p\times N},
\end{equation*}
where $d_u>0$ and $d_p>0$ are the number of degrees of freedom of the discrete velocity and pressure fields. The values of $d_u$ and $d_p$ depend on the mesh and on the choice of FEM space, but not on the choice of parameter $\bbf_{i, h}\in\mathcal{P}_{N, h}$.

The leading $n$ modes of variation ordered column-wise in $U_{u_h}$ and $U_{p_h}$ of the singular value decomposition of the matrices $M_{u_h}$ and $M_{p_h}$, respectively, define the subspaces $\mathbf{V}_{\text{POD},n}\times Q_{\text{POD},n}\subset  \bW\times Q$:
\begin{align*}
  M_{u_h} &\eqsvd U_{u_h} S_{u_h} V_{u_h},\qquad &U_{u_h}\in\mathbb{R}^{d_u\times n},\ S_{u_h}\in\mathbb{R}^{n\times n},\ V_{u_h}\in\mathbb{R}^{n\times d_u},\\
  M_{p_h} &\eqsvd U_{p_h} S_{p_h} V_{p_h},\qquad &U_{p_h}\in\mathbb{R}^{d_p\times n},\ S_{p_h}\in\mathbb{R}^{n\times n},\ V_{p_h}\in\mathbb{R}^{n\times d_p},
\end{align*}
where $S_{u_h}, S_{p_h}$ are the diagonal matrices of the singular values. 

We can define the linear operators $\Pi_{\mathbf{V}_n}:\bW\rightarrow \mathbf{V}_n$ and $\Pi_{Q_n}:Q\rightarrow Q_n$ as the orthogonal projectors onto the $n$-dimensional linear subspaces $\mathbf{V}_n \times Q_n \subset  \bW\times Q$ in $H^1(\Omega)^2$ and $L^2(\Omega)$, respectively. We will employ the notations $\Pi_{\mathbf{V}_{\text{POD},n}}$ and $\Pi_{Q_{\text{POD},n}}$, when $\mathbf{V}_n = \mathbf{V}_{\text{POD},n}$ and $Q_n = Q_{\text{POD},n}$, respectively.

Alternative strategies include weak Greedy procedures that rely on \textit{a posteriori} estimators~\cite{Veroy2005}. Despite several advantages, the main issue with these approaches is the availability of a reliable and efficient error estimator. Indeed, given the presence of nonlinear terms requiring the implementation of an \textit{a posteriori} estimators based on the Brezzi-Rappaz-Raviart theory \cite{BRR}, and the complexity of the parametrization chosen, here we only consider a POD-based reduction. Additionally, the computational cost can be further reduced using hyper-reduction strategies recovering the affine decomposition assumption, and thus the efficiency the whole strategy, e.g.\ via Empirical Interpolation Method (EIM)~\cite{BarraultEmpiricalInterpolationMethod2004,Canuto2009} or its discrete (DEIM)~\cite{ChaturantabutNonlinearModelReduction2010} and generalized (GEIM)~\cite{MadayEtAlGEIMStokes2015} versions.

%%%%%%%%%%%%%%%%%%%%%%%%%%%%%%%%%%%%%%%%%%%%%%%%%%%%%%%
\subsection{Upper bound on the Kolmogorov n-width}
\label{subsec:upperBoundKnw}
In order to show that the Kolmogorov $n$-width decay is bounded from above 
by the exponential rate derived in Theorem~\ref{thm:dnK}, 
we provide an estimate in terms of the discrete solutions of the incompressible Navier-Stokes equations. 

We want to estimate the Kolmogorov n-width of $\cK_N \subset H^1(\bbP)^2 \times L^2(\bbP)$, defined as:
\begin{align}
  \label{eq:KnWBd_INS}
    d_n(\cK_N, &H^1(\bbP)^2 \times L^2(\bbP))
    =\\
    &\inf_{\mathbf{V}_n \times Q_n \subset  \bW\times Q} 
    \max_{(\bu,p) \in \mathcal{K}_{N}}
    \| \bu - \Pi_{\mathbf{V}_n}\bu \|_{H^1(\bbP)^2} 
    +
    \| p   -  \Pi_{Q_n}p  \|_{L^2(\bbP)},\notag
\end{align}
where we took the maximum, since $\mathcal{K}_N$ has finite cardinality.

\begin{theorem}
  For each $N\in\mathbb{N}$, dimension of the parameter space $\mathcal{P}_N$, there exists an arbitrary small $\epsilon_{\mathcal{K}_{N}}>0$, such that the following upper bound to the Kolmogorov n-width is valid:
  \begin{align*}
    d_n(\cK_N, &H^1(\bbP)^2 \times L^2(\bbP))\leq \\ &\max_{(\mathbf{u}_h,p_h) \in \mathcal{K}_{N, h}}\| \mathbf{u}_h - \Pi_{\mathbf{V}_{\text{POD},n}}\mathbf{u}_h \|_{H^1(\bbP)} 
    +
    \| p_h   -  \Pi_{Q_{\text{POD},n}}p_h  \|_{L^2(\bbP)} + \epsilon_{\mathcal{K}_{N}}.
  \end{align*}
\end{theorem}
\begin{proof}
Notice that we can estimate the projection errors $\mathbf{u}_h - \Pi_{\mathbf{V}_{\text{POD},n}}$ and $p_h   -  \Pi_{Q_{\text{POD},n}}p_h$ only on discrete solutions $(\bu_h, p_h)\in\mathcal{K}_{N, h}\subset[\mathbb{P}_{\Gamma}^2(\mathcal{T}_h)]^2\times\mathbb{P}^1(\mathcal{T}_h)$ corresponding to the discrete counterpart of the parameter space $\mathcal{P}_{N,h}$ given by:
\begin{equation}
  \begin{split}
  \label{eq:fhspace1}
  \mathcal{P}_{N, h}=\Bigg\{\bbf_{i, h}\in [\mathbb{P}_{\Gamma}^2(\mathcal{T}_h)]^2\ :\ \int_{\Omega}\nabla\bbf_{i, h}\cdot\nabla\mathbf{g}_{h} &= \lambda_{n, h}\int_{\Omega}\bbf_{i, h}\cdot\nabla\mathbf{g}_{h},\\\; &\forall\mathbf{g}_{h}\in [\mathbb{P}_{\Gamma}^2(\mathcal{T}_h)]^2,\ i\in\{1,\dots, N\}\Bigg\}\;,
  \end{split}
\end{equation}
satisfying the additional small data hypothesis,$\ \| \bbf_{i, h}\|_{\bW^*}\leq \nu^2 \Ccoer^2/4\Ccont$.

We can now bound the inner part of the left-hand side in \eqref{eq:KnWBd_INS} as follows:
\begin{align*}
  &\max_{(\bu,p) \in \mathcal{K}_{N}}
    \| \bu - \Pi_{\mathbf{V}_n}\bu \|_{H^1(\bbP)^2} 
    +
    \| p   -  \Pi_{Q_n}p  \|_{L^2(\bbP)}\\
    &\leq \max_{(\bu,p) \in \mathcal{K}_{N}}\| \mathbf{u}_h - \Pi_{\mathbf{V}_n}\mathbf{u} \|_{H^1(\bbP)^2} 
    +
    \| p_h   -  \Pi_{Q_n}p_h  \|_{L^2(\bbP)} + \| \mathbf{u}_h - \ub \|_{H^1(\bbP)^2} 
    \\
    &\qquad\qquad+
    \| p_h   -  p  \|_{L^2(\bbP)}+ \| \Pi_{\mathbf{V}_n}\mathbf{u}_h - \Pi_{\mathbf{V}_n}\mathbf{u} \|_{H^1(\bbP)^2} 
    +
    \| \Pi_{Q_n}p_h   -  \Pi_{Q_n}p  \|_{L^2(\bbP)}\\
    &\leq \max_{(\mathbf{u}_h,p_h) \in \mathcal{K}_{N}}\| \mathbf{u}_h - \Pi_{\mathbf{V}_n}\mathbf{u}_h \|_{H^1(\bbP)^2} 
    +
    \| p_h   -  \Pi_{Q_n}p_h  \|_{L^2(\bbP)}\\
    &\quad+ \max_{(\mathbf{u},p) \in \mathcal{K}_{N, h}}\left(1 + \lVert \Pi_{\mathbf{V}_n}\rVert_{H^1(\bbP)^2}\right)\| \mathbf{u}_h - \ub \|_{H^1(\bbP)^2} 
    + \left(1 + \lVert \Pi_{Q_n}\rVert_{L^2(\bbP)}\right)
    \| p_h   -  p  \|_{L^2(\bbP)}\\
    &\leq\max_{(\mathbf{u}_h,p_h) \in \mathcal{K}_{N, h}}\| \mathbf{u}_h - \Pi_{\mathbf{V}_n}\mathbf{u}_h \|_{H^1(\bbP)^2} 
    +
    \| p_h   -  \Pi_{Q_n}p_h  \|_{L^2(\bbP)} + \epsilon_{\mathcal{K}_N},
\end{align*}
where we have made the assumption that the discrete solutions approximate at an arbitrary small tolerance $\epsilon_{\mathcal{K}_N}>0$ the continuous solutions in $\mathcal{K}_N$:
\begin{equation*}
  \max_{(\mathbf{u},p) \in \mathcal{K}_{N}}\left(1 + \lVert \Pi_{\mathbf{V}_n}\rVert_{H^1(\bbP)^2}\right)\| \mathbf{u}_h - \ub \|_{H^1(\bbP)^2} 
    + \left(1 + \lVert \Pi_{Q_n}\rVert_{L^2(\bbP)}\right)
    \| p_h   -  p  \|_{L^2(\bbP)}\lesssim \epsilon_{\mathcal{K}_N}\;.
\end{equation*}
Restricting ourselves to the linear subspaces $\mathbf{V}_{\text{POD},n}\times Q_{\text{POD},n}\subset  \bW\times Q$ obtained from the proper orthogonal decomposition of the discrete solutions in $\mathcal{K}_{N, h}$, as described in subsection~\ref{subsec: pod}, we find the following upper bound on the Kolmogorov n-width:
\begin{align*}
  d_n(&\cK_N, H^1(\bbP)^2 \times L^2(\bbP))
  =\\
  &=\inf_{\mathbf{V}_n \times Q_n \subset  \bW\times Q} 
  \max_{(\bu,p) \in \mathcal{K}_{N}}
  \| \bu - \Pi_{\mathbf{V}_n}\bu \|_{H^1(\bbP)^2} 
  +
  \| p   -  \Pi_{Q_n}p  \|_{L^2(\bbP)}\\
  &\leq\inf_{\mathbf{V}_n \times Q_n \subset  \bW\times Q} 
  \max_{(\mathbf{u}_h,p_h) \in \mathcal{K}_{N, h}}\| \mathbf{u}_h - \Pi_{\mathbf{V}_n}\mathbf{u}_h \|_{H^1(\bbP)} 
    +
    \| p_h   -  \Pi_{Q_n}p_h  \|_{L^2(\bbP)} + \epsilon_{\mathcal{K}_N}\\
  &\leq
  \max_{(\mathbf{u}_h,p_h) \in \mathcal{K}_{N, h}}\| \mathbf{u}_h - \Pi_{\mathbf{V}_{\text{POD},n}}\mathbf{u}_h \|_{H^1(\bbP)} 
    +
    \| p_h   -  \Pi_{Q_{\text{POD},n}}p_h  \|_{L^2(\bbP)} + \epsilon_{\mathcal{K}_{N}}.
\end{align*}
\end{proof}
Our final purpose is to show experimentally that there exist $b,C > 0$ such that:
\begin{equation}
  \label{eq:ubKnW}
  \max_{(\mathbf{u}_h,p_h) \in \mathcal{K}_{N, h}}\| \mathbf{u}_h - \Pi_{\mathbf{V}_{\text{POD},n}}\mathbf{u}_h \|_{H^1(\bbP)} 
    +
    \| p_h   -  \Pi_{Q_{\text{POD},n}}p_h  \|_{L^2(\bbP)}\leq C\exp( - b n^{1/3}),
\end{equation}
in this way, assuming that $\epsilon_{\mathcal{K}_{N}}\ll C\exp( - b n^{1/3})$, we will show numerically the decay rate proved in Theorem~\ref{thm:dnK}.

\subsection{Numerical framework}
\label{subsec:numframe}

As discussed in Section \ref{subsec:discreteINSFEM}, we solve for the INS exploiting the LBB-stable Lagrangian finite element spaces $\mathbb{P}^2-\mathbb{P}^{1}$, on two different domains: one for the backward-facing step problem, 
and one for the flow past a cube in a channel problem~\cite{Couplet2005}, as depicted in Figure \ref{fig:domains}, left and right respectively. The first geometry models a sudden-expansion channel, 
while the second one is characterized by a cube (in 2D a square) centered in $(0.2, 0.2)$ with side length $0.05$. 

Two unstructured meshes have been exploited to discretize the geometries, 
consisting of $3295$ nodes and $6320$ cells for the backward facing step, 
and $1568$ nodes and $2928$ triangular cells for the flow past a cube.

\begin{figure}[!hb]
  \centering
  \begin{tikzpicture}[scale=1.4]
    \draw[fill=blue!50!green] (0,0.4) -- (0.8,0.4) -- (0.8,0) -- (4,0) -- (4,1) -- (0,1) -- cycle;
    \node[below] at (0,0.4) {\scriptsize{(0, 2)}};
    \node[above] at (0,1) {\scriptsize{(0, 5)}};
    \node[above] at (4,1) {\scriptsize{(22, 5)}};
    \node[below] at (4,0) {\scriptsize{(22, 0)}};
    \node[below] at (0.8,0) {\scriptsize{(4, 0)}};
    % \draw[->,>=stealth] (-0.3,0.9) -- (0.,0.9);
    % \draw[->,>=stealth] (-0.3,0.8) -- (0.,0.8);
    % \draw[->,>=stealth] (-0.3,0.7) -- (0.,0.7);
    % \draw[->,>=stealth] (-0.3,0.6) -- (0.,0.6);
    % \draw[->,>=stealth] (-0.3,0.5) -- (0.,0.5);
    % \node[below] at (0.5,0.3) {\scriptsize{$\Gamma_\text{in}$}};
    \node[below] at (2,0) {\scriptsize{$\Gamma_\text{w}$}};
    \node[above] at (2,1) {\scriptsize{$\Gamma_\text{w}$}};
    \node[left,white] at (1.2,0.2) {\scriptsize{$\Gamma_\text{v}$}};
    \node[left] at (0,0.7) {\scriptsize{$\Gamma_\text{in}$}};
    \node[right] at (4.,0.5) {\scriptsize{$\Gamma_\text{out}$}};
    \node[right,white] at (0.3,0.55) {\scriptsize{$\Gamma_\text{h}$}};
    % \node[white,right] at (0,0.7) {{$\Gamma_{in}$}};
    % \node[white,left] at (4,0.5) {{$\Gamma_{out}$}};
    \node[white] at (2,0.5) {\Large{$\Omega$}};
    \end{tikzpicture}
    \begin{tikzpicture}[scale=2.95]
      \draw[fill=blue!50!blue] (0,0) -- (2.2,0) -- (2.2,0.4) -- (0,0.4) -- cycle;
      \filldraw[color=black, fill=white](0.15,0.15) rectangle (0.25,0.25);
      \node[below] at (0,0) {\scriptsize{(0, 0)}};
      %\node[below] at (0.65,0.4) {\scriptsize{(4, 2)}};
      \node[above] at (0,0.4) {\scriptsize{(0, 0.4)}};
      \node[above] at (2.2,0.4) {\scriptsize{(2.2, 0.4)}};
      \node[below] at (2.2,0) {\scriptsize{(2.2, 0)}};
      % \node[white,right,below] at (0.15,0.15) {\scriptsize{(0.15, 0.15)}};
      % \node[white,right,above] at (0.25,0.25) {\scriptsize{(0.25, 0.25)}};
      % \draw[->,>=stealth] (-0.3,0.9) -- (0.,0.9);
      % \draw[->,>=stealth] (-0.3,0.7) -- (0.,0.7);
      % \draw[->,>=stealth] (-0.3,0.5) -- (0.,0.5);
      % \draw[->,>=stealth] (-0.3,0.3) -- (0.,0.3);
      % \draw[->,>=stealth] (-0.3,0.1) -- (0.,0.1);
      \node[left] at (0,0.2) {\scriptsize{$\Gamma_\text{in}$}};
      \node[right] at (2.2,0.2) {\scriptsize{$\Gamma_\text{out}$}};
      \node[above] at (1.1,0.4) {\scriptsize{$\Gamma_\text{w}$}};
      \node[below] at (1.1,0) {\scriptsize{$\Gamma_\text{w}$}};
      \node[white] at (1.1,0.2) {\large{$\Omega$}};
      \node[white,below] at (0.35,0.28) {\scriptsize{$\Gamma_\text{c}$}};
      \end{tikzpicture}
  \caption{Geometrical set-up of the INS problems with Dirichlet homogeneous boundary conditions: backward facing step (left) and flow past a cube (right).}
  \label{fig:domains}
  \end{figure}
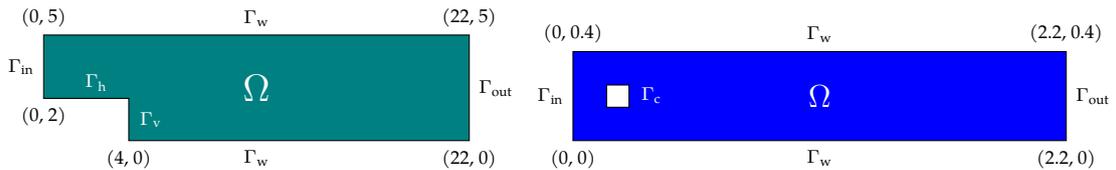

In both cases, 
we employ the weak saddle-point formulation described in subsection~\ref{subsec:discreteINSFEM} 
and homogeneous Dirichlet boundary conditions for the velocity space.
The parameter is the source term that belongs to the parameter space $\mathcal{P}_{N, h}$. 
We will consider $N=|\mathcal{K}_{N, h}|=|\mathcal{P}_{N, h}|$ and $\nu=1$.

The approximate coercivity and continuity constants are 
$\Ccoerh=66$, $\Cconth=526633$ for the flow past a cube test case 
and 
$\Ccoerh=0.448$, $\Cconth=15475$ for the backward step test case. 
The source terms $\mathbf{f}_{i, h}\in\mathcal{P}_{N, h}$ are rescaled to satisfy the small data assumption, i.e.\ such that:
\begin{equation}
  \lVert \mathbf{f}_{i, h}\rVert_{\bW^*_h }=\nu^2 \frac{\Ccoerh^2}{4\Cconth}, \qquad\forall\mathbf{f}_{i, h}\in\mathcal{P}_{N, h}.
\end{equation}

To estimate the KnW of the solution set, we evaluate the upper bound derived from the previous subsection, 
Equation~\eqref{eq:ubKnW}, for different number of POD basis $n\leq N$ computed from the solutions $\{\mathbf{u}_h(\mathbf{f}_{i, h}), p_h(\mathbf{f}_{i, h})\}_{n=1}^{N}$  of the INS equations~\eqref{eq:INS} with source terms in $\mathbf{f}_{i,h}\in\mathcal{K}_{N, h},\ \forall i\in\{1, \dots, N\}$:
\begin{align}
  \max_{(\mathbf{u}_h,p_h) \in \mathcal{K}_{N, h}}\| \mathbf{u}_h &- \Pi_{\mathbf{V}_{\text{POD},n}}\mathbf{u}_h \|_{H^1(\bbP)} 
    +
    \| p_h   -  \Pi_{Q_{\text{POD},n}}p_h  \|_{L^2(\bbP)}\leq \epsilon_u+\epsilon_p\notag\\
    &\epsilon_u = \max_{\mathbf{u}_h \in \pi_1(\mathcal{K}_{N, h})}\| \mathbf{u}_h - \Pi_{\mathbf{V}_{\text{POD},n}}\mathbf{u}_h \|_{H^1(\bbP)}\label{eq:epsilonu},\\
    &\epsilon_p = \max_{p_h \in \pi_2(\mathcal{K}_{N, h})}\| p_h   -  \Pi_{Q_{\text{POD},n}}p_h  \|_{L^2(\bbP)}.\label{eq:epsilonp}
\end{align}

Since the exponential bound of Theorem~\ref{thm:dnK} is valid for all boundary conditions in Equation~\ref{eq:NS}, we present now several numerical results corresponding to theoretically admissible boundary conditions for the two benchmarks.

%%%%%%%%%%%%%%%%%%%%%%%%%%%%%%%%%%%%%%%%%%%%%%%%%%%%%%%%%%%%%%%%%%%%%%%%%%%%%%
\subsubsection{Dirichlet homogeneous boundary conditions}
\label{subsec:dir}
In this section, we consider the homogeneous conditions on the Dirichlet boundary $\Gamma_D$ defined as the whole boundary of the two domains $\Omega$, that is for the backward facing step problem one has $\Gamma_D$ = $\Gamma_\text{in} \cup \Gamma_\text{w} \cup \Gamma_\text{out} \cup \Gamma_\text{v} \cup \Gamma_\text{h}$, while for the flow past a cube geometry it holds $\Gamma_D = \Gamma_\text{in} \cup \Gamma_\text{w} \cup \Gamma_\text{out} \cup \Gamma_\text{c}$.
The velocity and pressure fields along with their corresponding source term 
$\mathbf{f}_{i, h}\in\mathcal{P}_{N, h}$ for $i=2$ are shown in Figure~\ref{fig:vpbackcube1} 
for the flow past a cube and backward step, top and bottom respectively.

\begin{figure}[tph!]
  \centering
  \includegraphics[width=0.49\textwidth, trim={0 40 0 105}, clip]{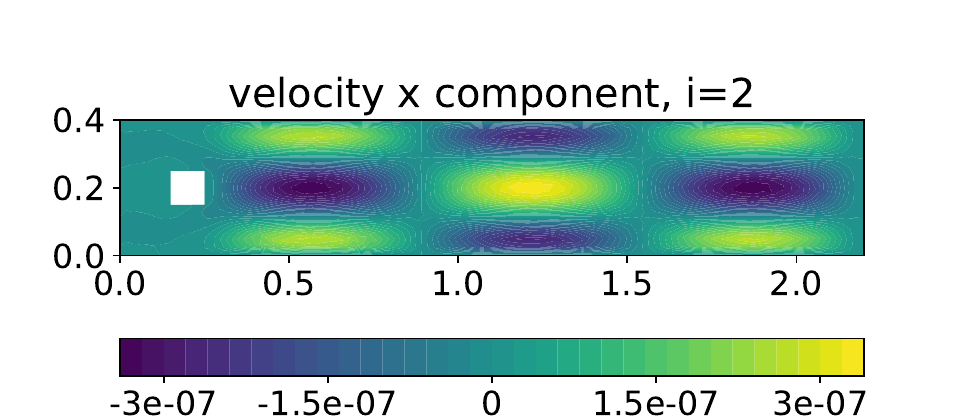}
  \includegraphics[width=0.49\textwidth, trim={0 40 0 105}, clip]{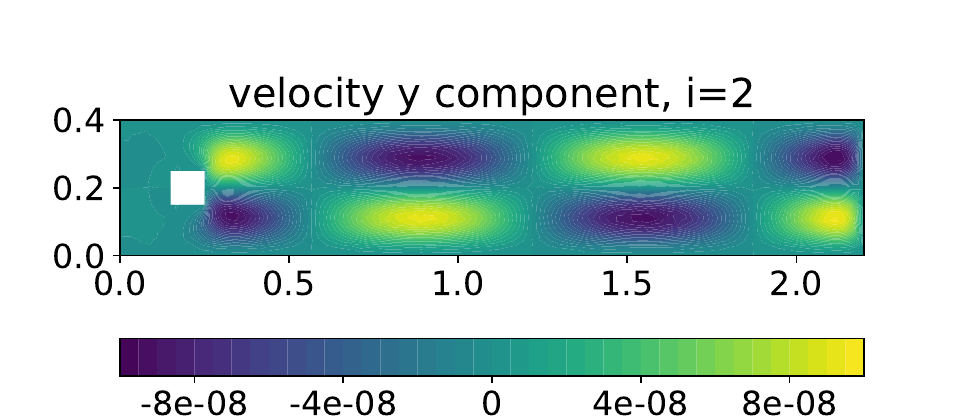}
  \includegraphics[width=0.49\textwidth, trim={0 40 0 105}, clip]{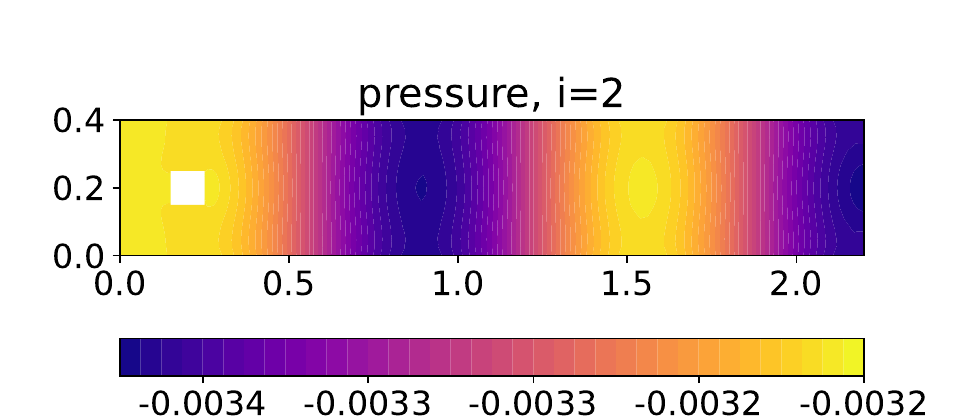}
  \includegraphics[width=0.49\textwidth, trim={0 40 0 105}, clip]{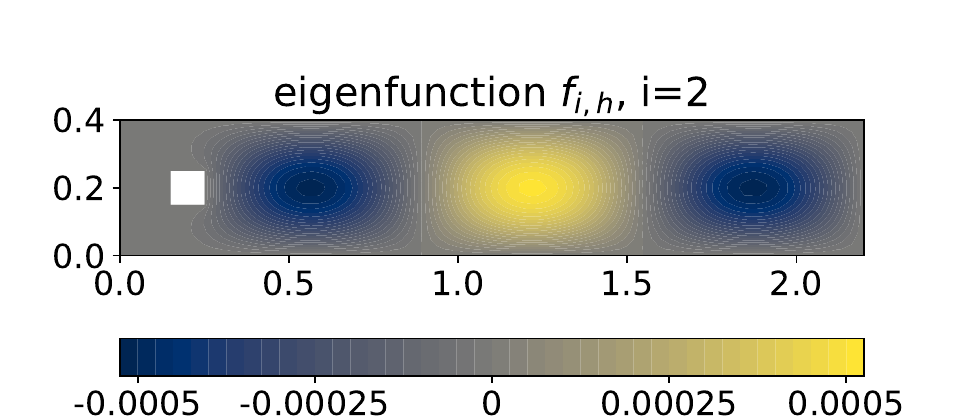}\\[0.5cm]
  \includegraphics[width=0.49\textwidth, trim={0 40 0 105}, clip]{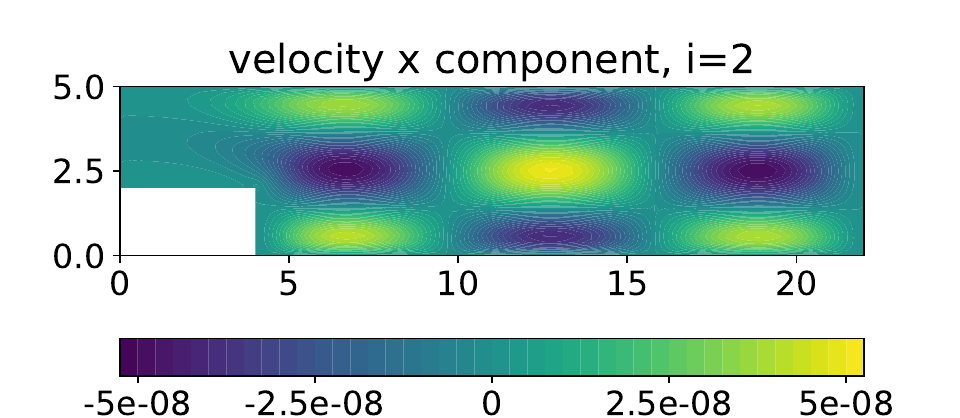}
  \includegraphics[width=0.49\textwidth, trim={0 40 0 105}, clip]{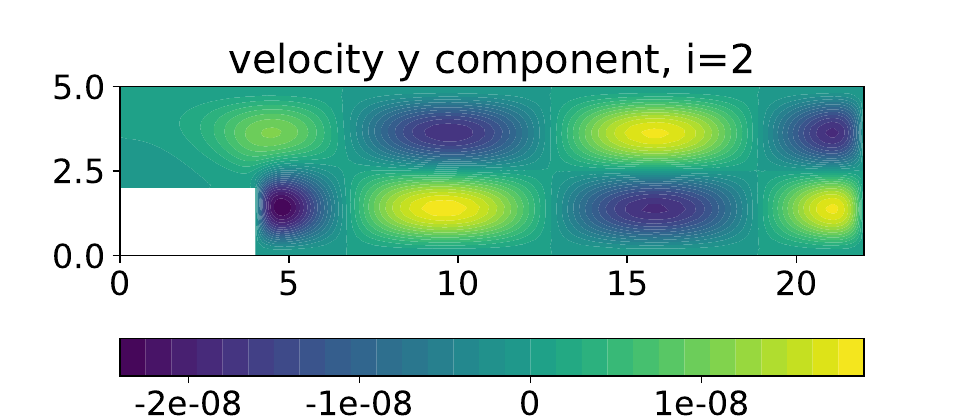}
  \includegraphics[width=0.49\textwidth, trim={0 40 0 105}, clip]{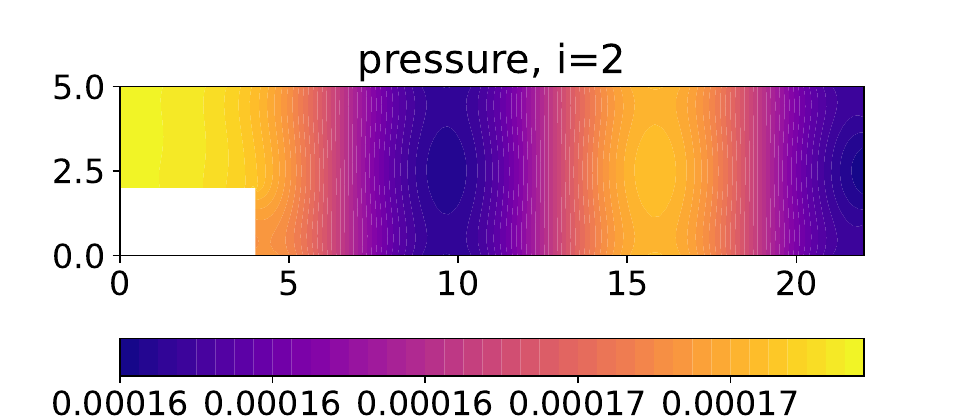}
  \includegraphics[width=0.49\textwidth, trim={0 40 0 105}, clip]{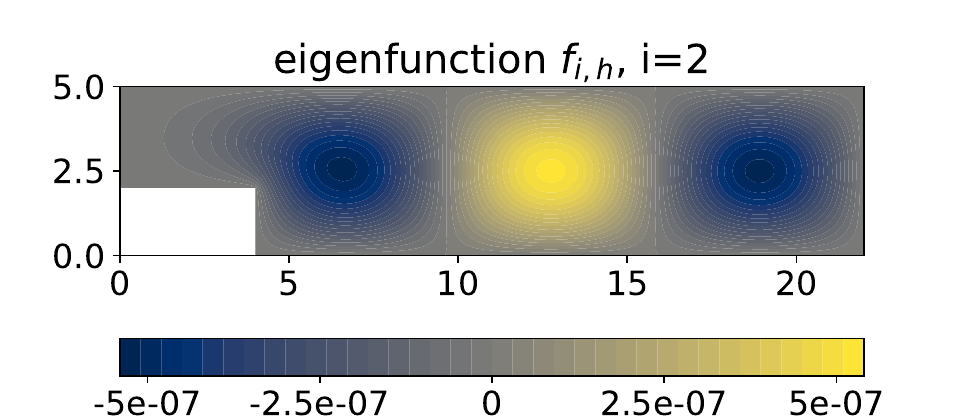}
  \caption{Snapshots of velocity and pressure fields along with corresponding eigenfunction $\mathbf{f}_{i,h}$ 
           of the Dirichlet-Laplace operator, for the INS with homogeneous Dirichlet boundary conditions: flow past a cube (top) and backward facing step (bottom).
}
  \label{fig:vpbackcube1}
\end{figure}

The results of the upper bounds $\epsilon_u$ and $\epsilon_p$ of the Kolmogorov $n$-width for homogeneous Dirichlet boundary conditions are shown in Figure~\ref{fig:dirichlet}. We investigate the behaviour of the decay while varying the number $n$ of POD modes, 
for different amount of sampled solutions $N\in\{250, 500, 750\}$, $n\leq N$. 
KnWs are numerically estimated with respect to a numerical, high-fidelity reference solution computed on a fixed mesh discretization. 
For both benchmarks, velocity and pressure fields exhibit the predicted exponential decay rate $\exp(- b n^{1/3}-c)$, until the asymptotic super-exponential behaviour is reached (due to the standard self-convergence of the POD).

\begin{figure}[tph!]
  \centering
  \includegraphics[width=1\textwidth]{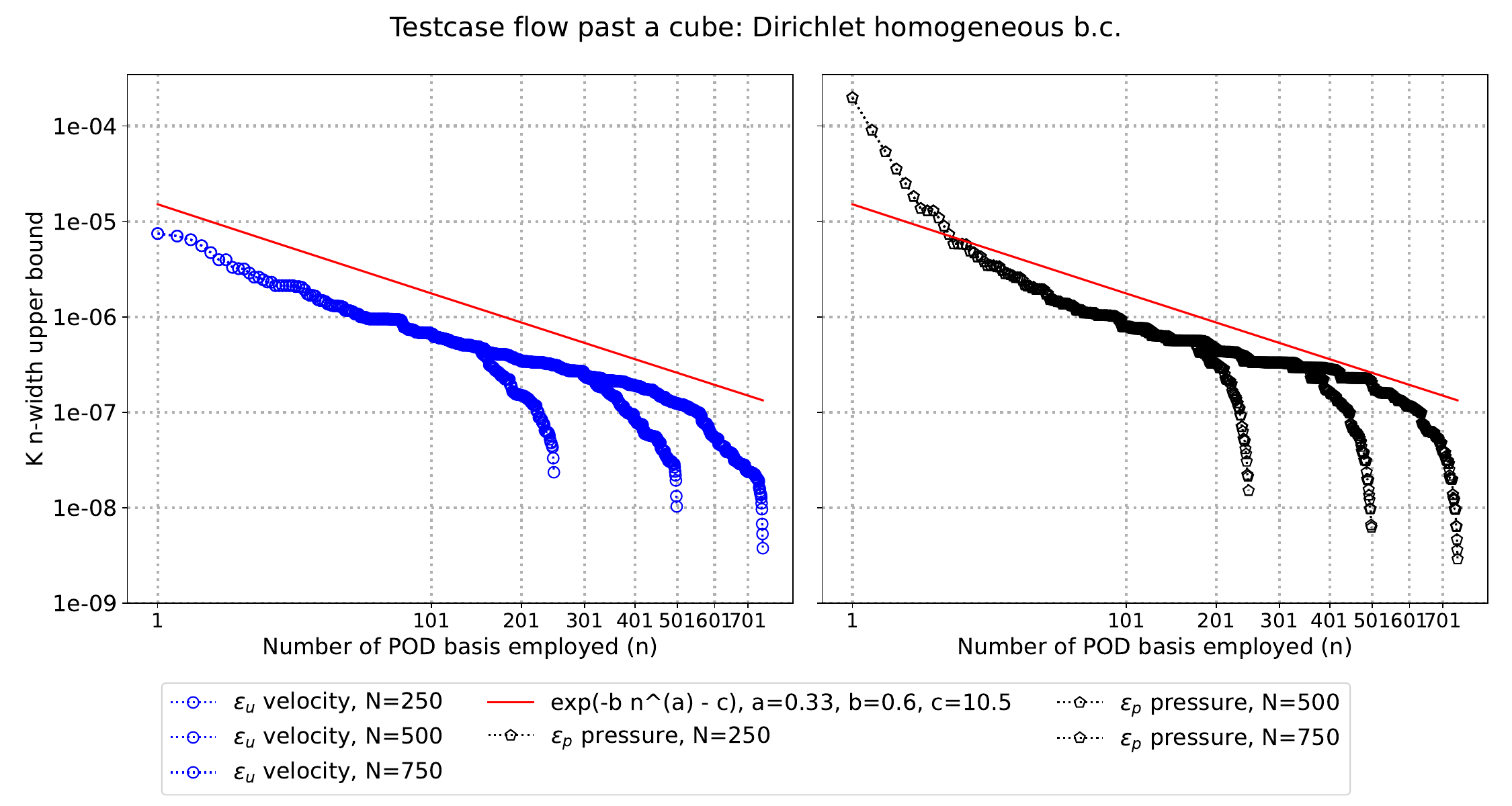}
  \includegraphics[width=1\textwidth]{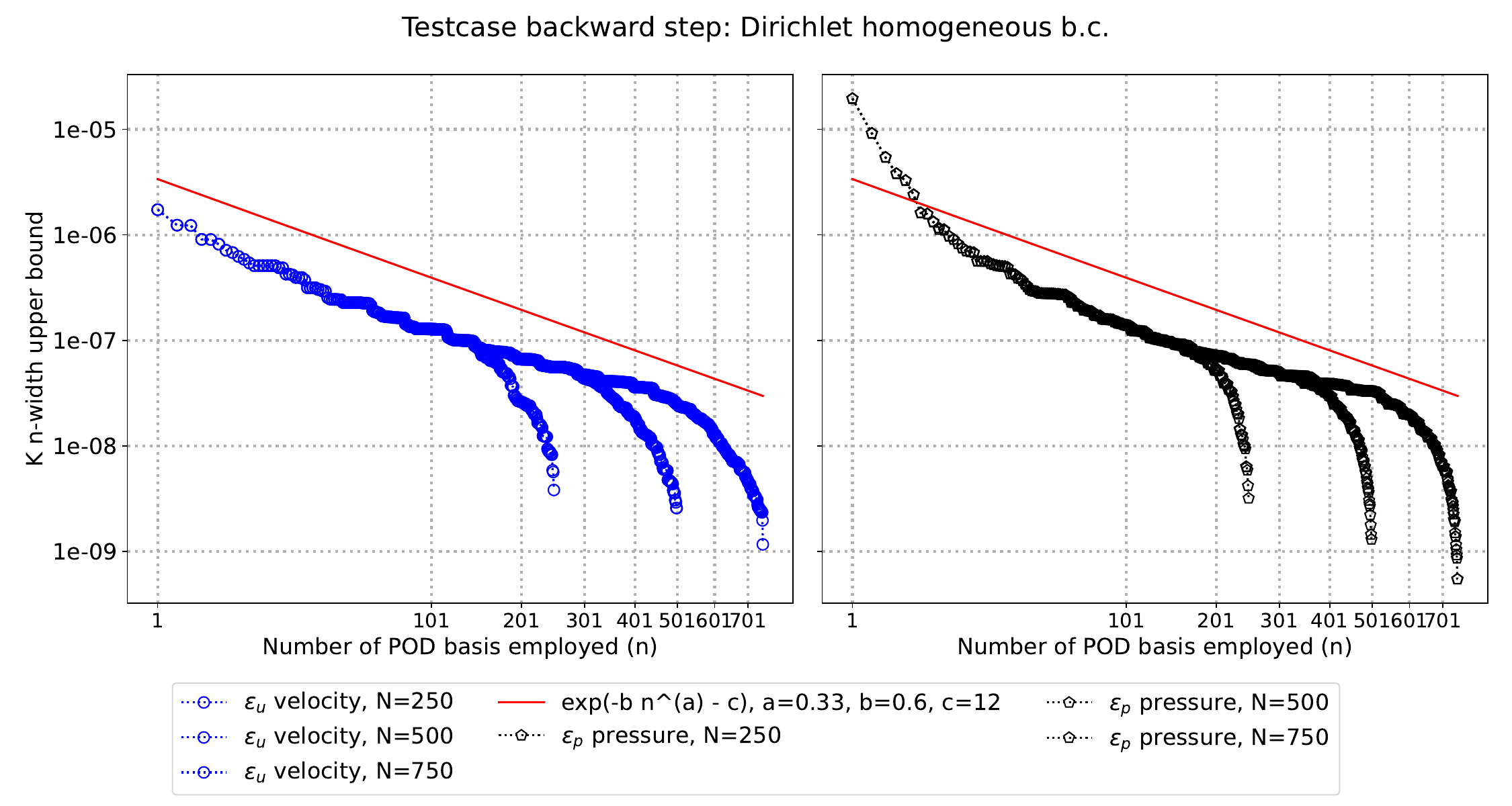}
  \caption{Numerically estimated upper bounds $\epsilon_u$ and $\epsilon_p$ of the Kolmogorov $n$-width for homogeneous Dirichlet boundary conditions varying the amount of snapshots $N$: flow past a cube (top) and backward facing step (bottom). The asymptotic super-exponential behaviour is due to self-convergence to the high-fidelity numerical reference solution.}
  \label{fig:dirichlet}
\end{figure}

To discuss the influence of the kinematic viscosity $\nu$ and the discretization of the high-fidelity snapshots on the decay of the KnW, we depict the evolution of the upper bounds $\epsilon_u$ and $\epsilon_p$ in Figure~\ref{fig:nudirichlet} and Figure~\ref{fig:meshdirichlet}, respectively. For both investigations, we vary the number of POD modes $n\leq N$, and obtain consistent results up to the self-convergence phenomenon. 
In Figure \ref{fig:nudirichlet}, we consider different values of the kinematic viscosity $\nu\in\{0.1, 0.5, 1\}$, and the same decay rate has been observed, consistently to the theoretical results of Theorem~\ref{thm:dnK}.

\begin{figure}[tph!]
  \centering
  \includegraphics[width=1\textwidth]{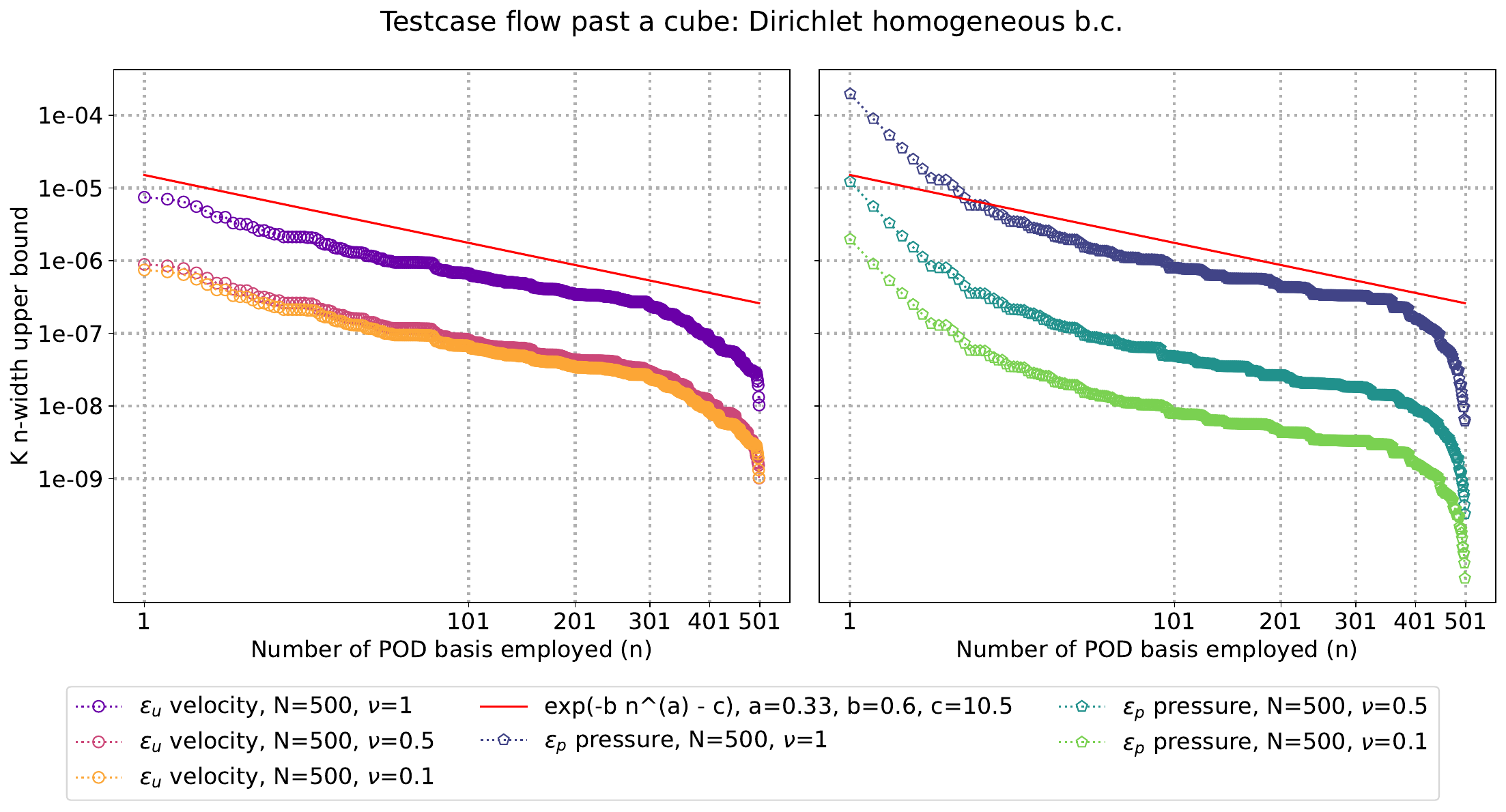}
  \includegraphics[width=1\textwidth]{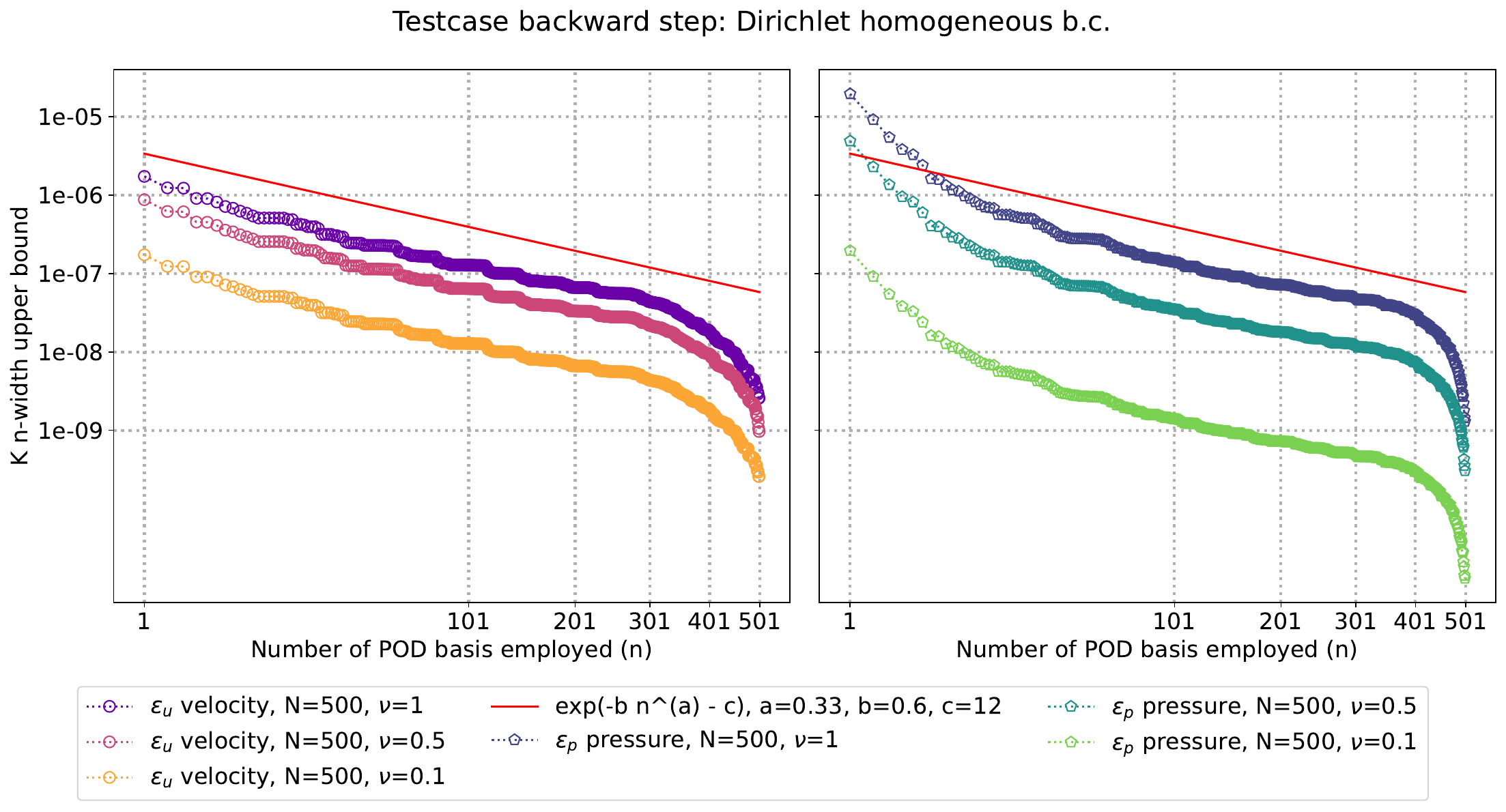}
  \caption{Numerically estimated upper bounds $\epsilon_u$ and $\epsilon_p$ of the Kolmogorov $n$-width for homogeneous Dirichlet boundary conditions varying the viscosity $\nu\in\{0.1, 0.5, 1\}$: flow past a cube (top) and backward facing step (bottom). The asymptotic super-exponential behaviour is due to self-convergence to the high-fidelity numerical reference solution.
%   Dirichlet homogeneous boundary conditions 
% Top: flow past a cube in channel. Bottom: backward-facing step. 
% Upper bounds $\epsilon_u$ and $\epsilon_p$ of the Kolmogorov n-width, 
% varying the kinematic viscosity $\nu\in\{0.1, 0.5, 1\}$, 
% the number $n$ of POD modes from $1$ to $500$, while
% keeping fixed the number of parameters $N=500$, $n\leq N$. KnWs are numerically estimated with respect to a numerical, high-fidelity reference solution. The asymptotic behaviour is due to self-convergence.
}
  \label{fig:nudirichlet}
\end{figure}

Similarly, in Figure~\ref{fig:meshdirichlet}, we show the decay of the KnW for different mesh resolutions, a fine and a coarse discretizations, while keeping fixed the number of parameters $N=500$ (top) and $N=750$ (bottom). Also in this context, once the discretization is fine enough to resolve all the fluid scales, the decay of the KnW is shown to be not affected by the mesh resolution. 

\begin{figure}[tph!]
  \centering
  \includegraphics[width=1\textwidth]{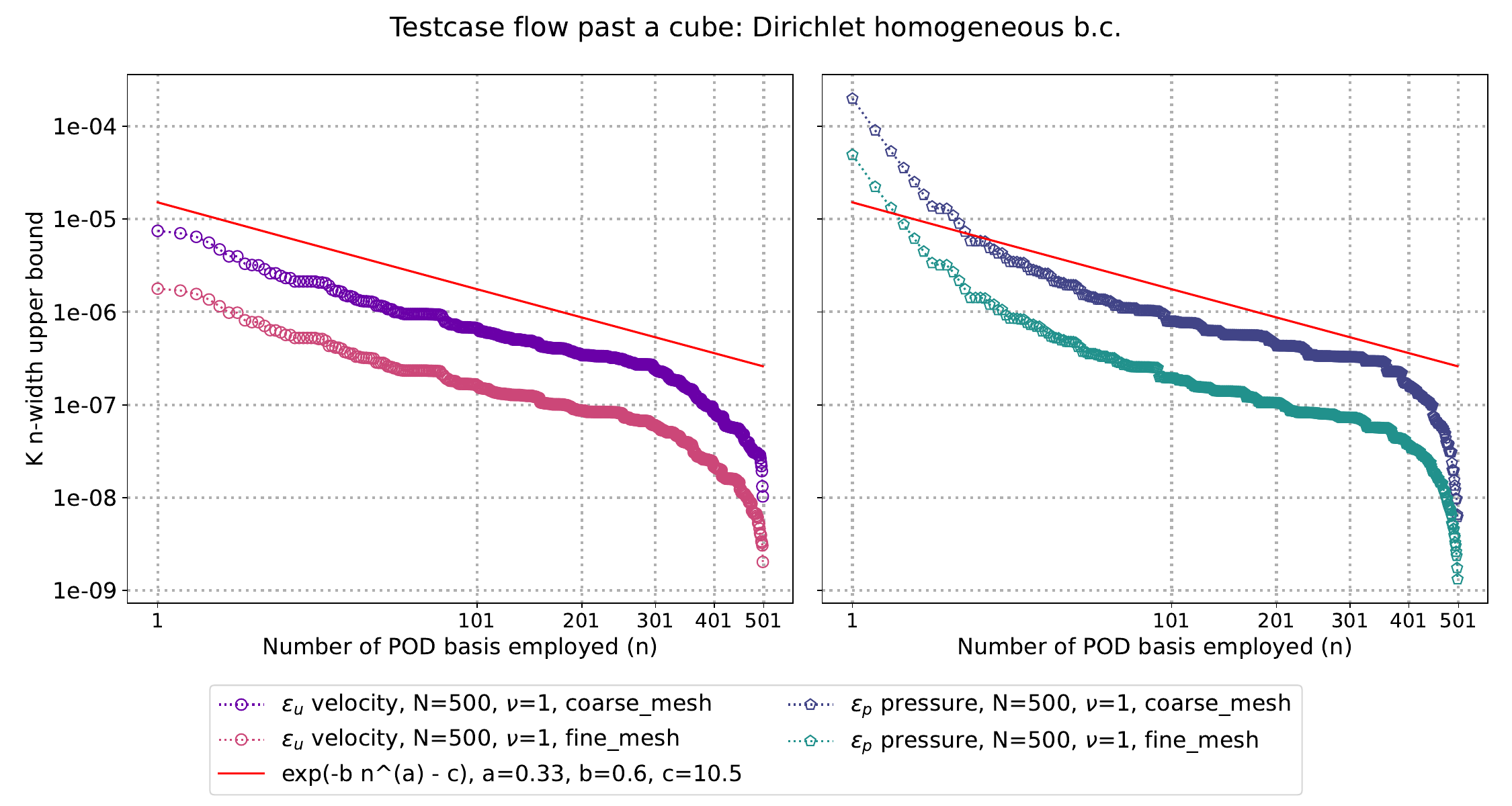}
  \includegraphics[width=1\textwidth]{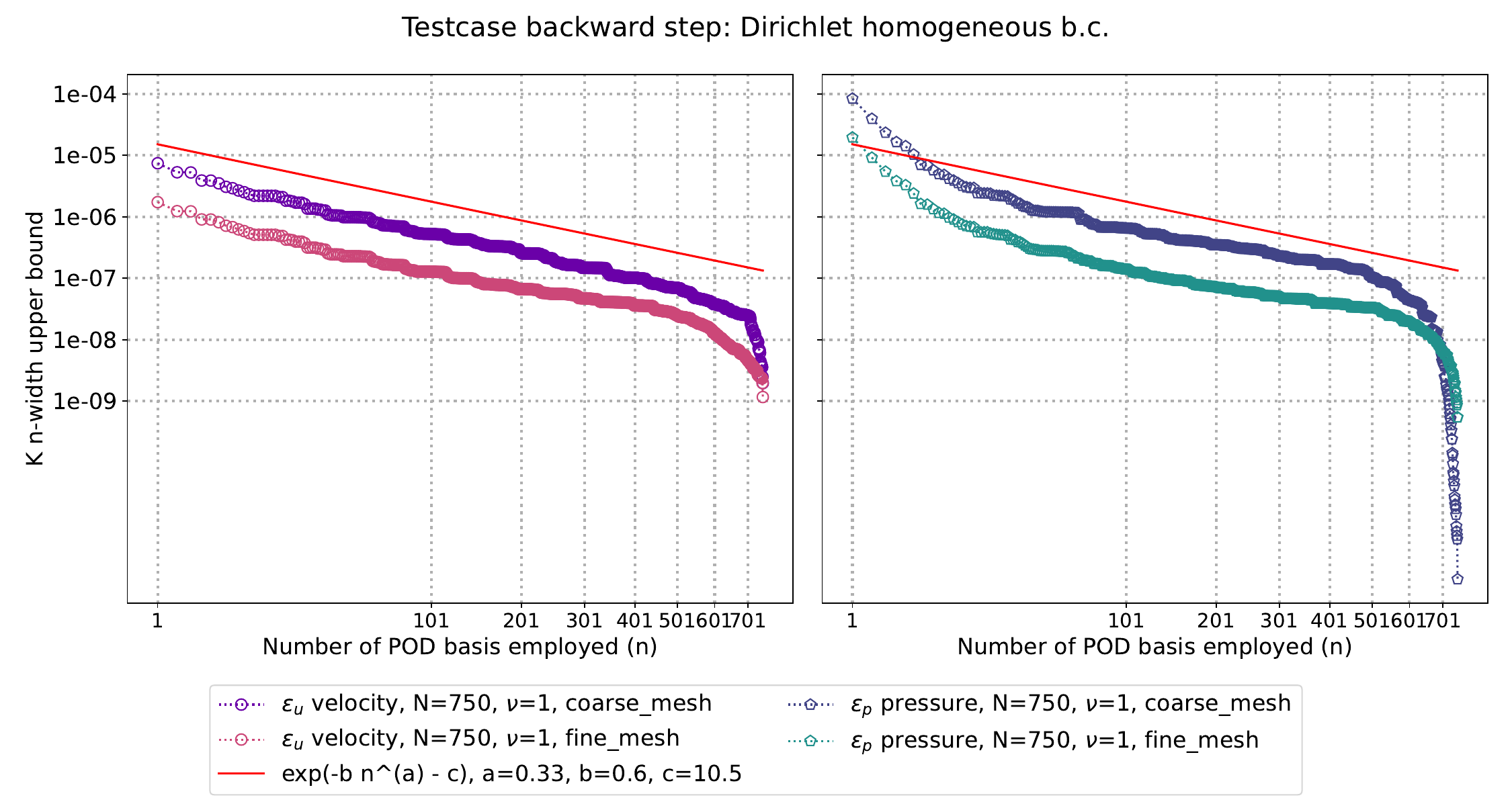}
  \caption{Numerically estimated upper bounds $\epsilon_u$ and $\epsilon_p$ of the Kolmogorov $n$-width for homogeneous Dirichlet boundary conditions varying mesh resolutions: flow past a cube (top) and backward facing step (bottom). The asymptotic super-exponential behaviour is due to self-convergence to the high-fidelity numerical reference solution.
% Homogeneous Dirichlet boundary conditions 
% Top: flow past a cube in channel. Bottom: backward-facing step. 
% Upper bounds $\epsilon_u$ and $\epsilon_p$ of the Kolmogorov $n$-width, 
% varying the number $n$ of POD modes and the mesh, 
% while keeping the number of parameters fixed at $N=500$ (top) and at $N=750$ (bottom), $n\leq N$. KnWs are numerically estimated with respect to a numerical, high-fidelity reference solution. The asymptotic behaviour is due to self-convergence.
}
  \label{fig:meshdirichlet}
\end{figure}

\subsubsection{Mixed Dirichlet-Neumann homogeneous boundary conditions}
\label{subsec:mixed}

In this section, we report the decay of the upper bounds on the velocity and pressure reconstruction errors, $\epsilon_u$  and $\epsilon_p$, respectively in Equations~\eqref{eq:epsilonu} and  \eqref{eq:epsilonp}, for homogeneous Dirichlet-Neumann boundary conditions.
In particular, following the notation of Figure ~\ref{fig:domains}, we define the Dirichlet boundary for the two benchmarks, respectively as $\Gamma_D = \Gamma_\text{w} \cup \Gamma_\text{v} \cup \Gamma_\text{h}$ and $\Gamma_D = \Gamma_\text{w} \cup \Gamma_\text{c}$, and the Neumann boundary as $\Gamma_N = \Gamma_\text{in} \cup \Gamma_\text{out}$ for both geometries.
The velocity and pressure fields along with their corresponding source term 
$\mathbf{f}_{i, h}\in\mathcal{P}_{N, h}$ for $i=2$ are shown in Figure~\ref{fig:vpbackcube2} 
for the flow past a cube and backward step, top and bottom respectively.

\begin{figure}[tph!]
  \centering
  \includegraphics[width=0.49\textwidth, trim={0 40 0 105}, clip]{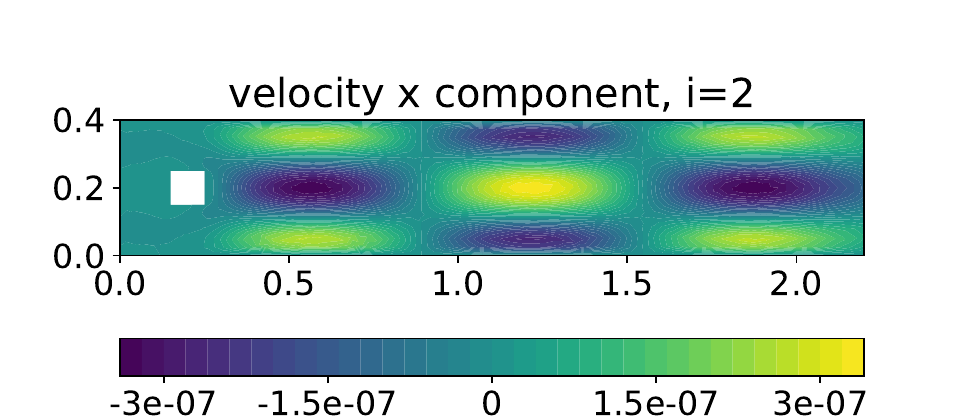}
  \includegraphics[width=0.49\textwidth, trim={0 40 0 105}, clip]{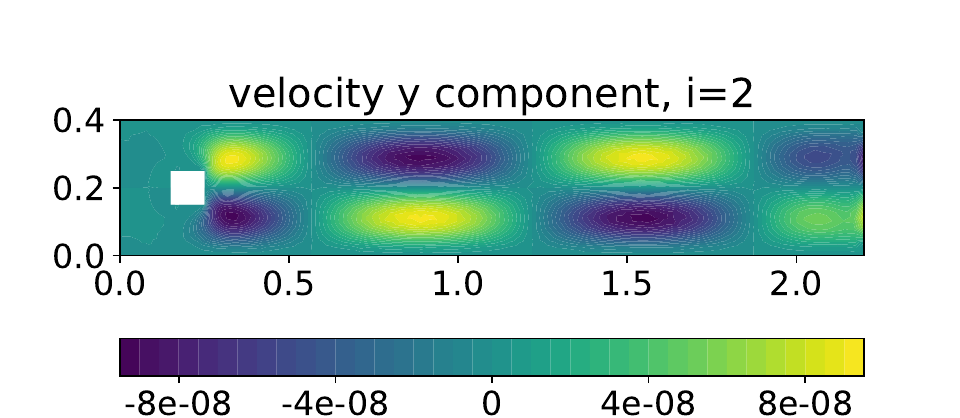}
  \includegraphics[width=0.49\textwidth, trim={0 40 0 105}, clip]{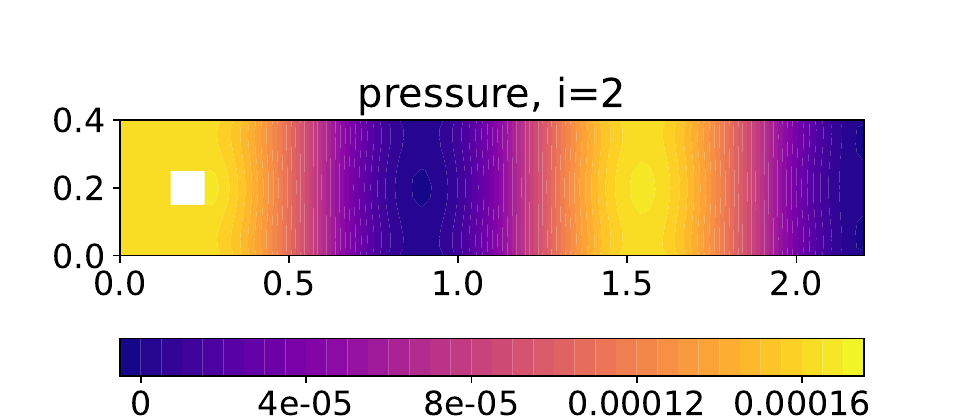}
  \includegraphics[width=0.49\textwidth, trim={0 40 0 105}, clip]{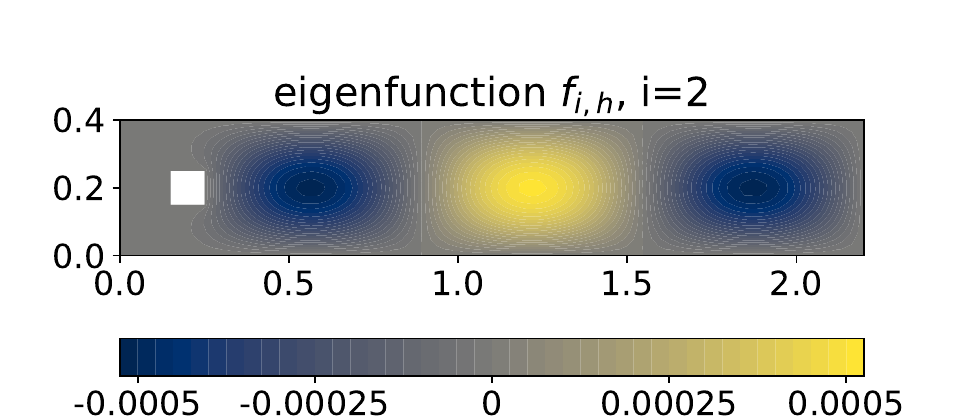}\\[0.5cm]
  \includegraphics[width=0.49\textwidth, trim={0 40 0 105}, clip]{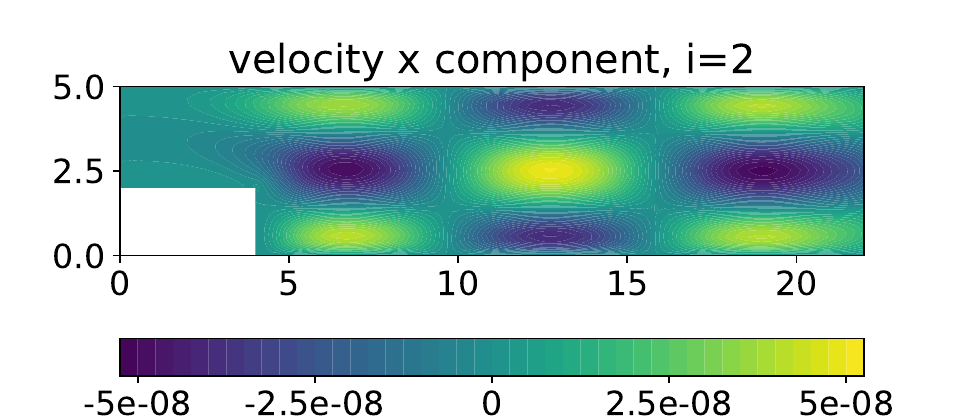}
  \includegraphics[width=0.49\textwidth, trim={0 40 0 105}, clip]{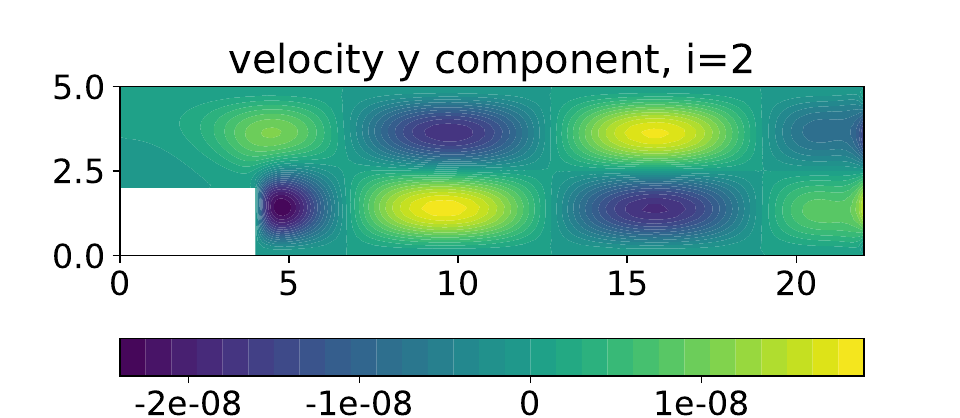}
  \includegraphics[width=0.49\textwidth, trim={0 40 0 105}, clip]{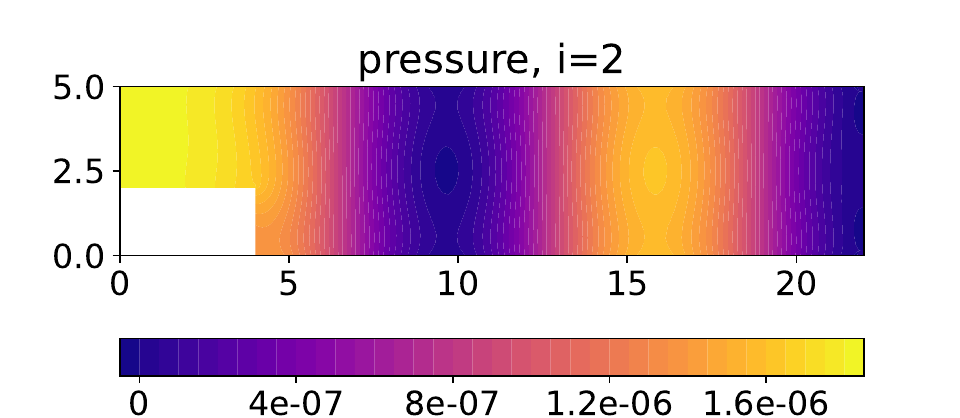}
  \includegraphics[width=0.49\textwidth, trim={0 40 0 105}, clip]{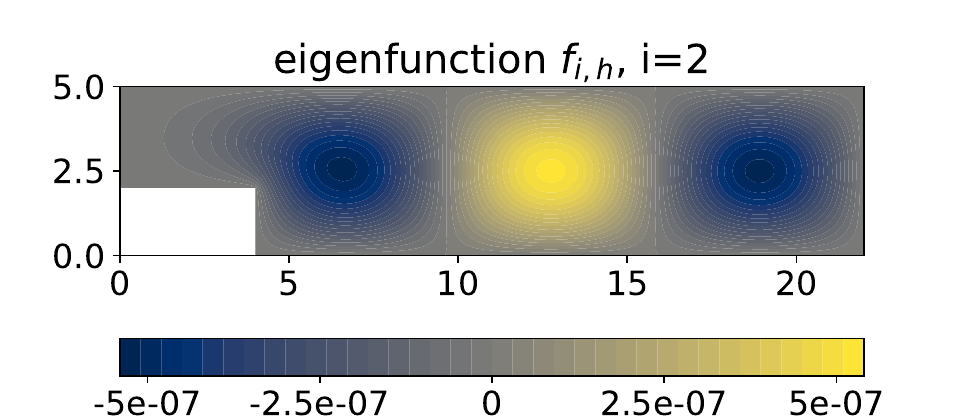}
  \caption{Snapshots of velocity and pressure fields along with corresponding eigenfunction $\mathbf{f}_{i,h}$ of the Dirichlet-Laplace operator, for the INS with mixed homogeneous Dirichlet-Neumann boundary conditions: flow past a cube (top) and backward facing step (bottom).}
  \label{fig:vpbackcube2}
\end{figure}

The results of the upper bounds $\epsilon_u$ and $\epsilon_p$ of the Kolmogorov $n$-width for homogeneous Dirichlet-Neumann boundary conditions are reported in Figure~\ref{fig:mixed}. We plot the behaviour of the upper bounds $\epsilon_u$ and $\epsilon_p$ while varying the amount of parameters $N\in\{250, 500, 750\}$. Also here,  KnWs are numerically estimated with respect to a numerical, high-fidelity reference solution and the self-convergence behaviour is observed. 

\begin{figure}[tph!]
  \centering
  \includegraphics[width=1\textwidth]{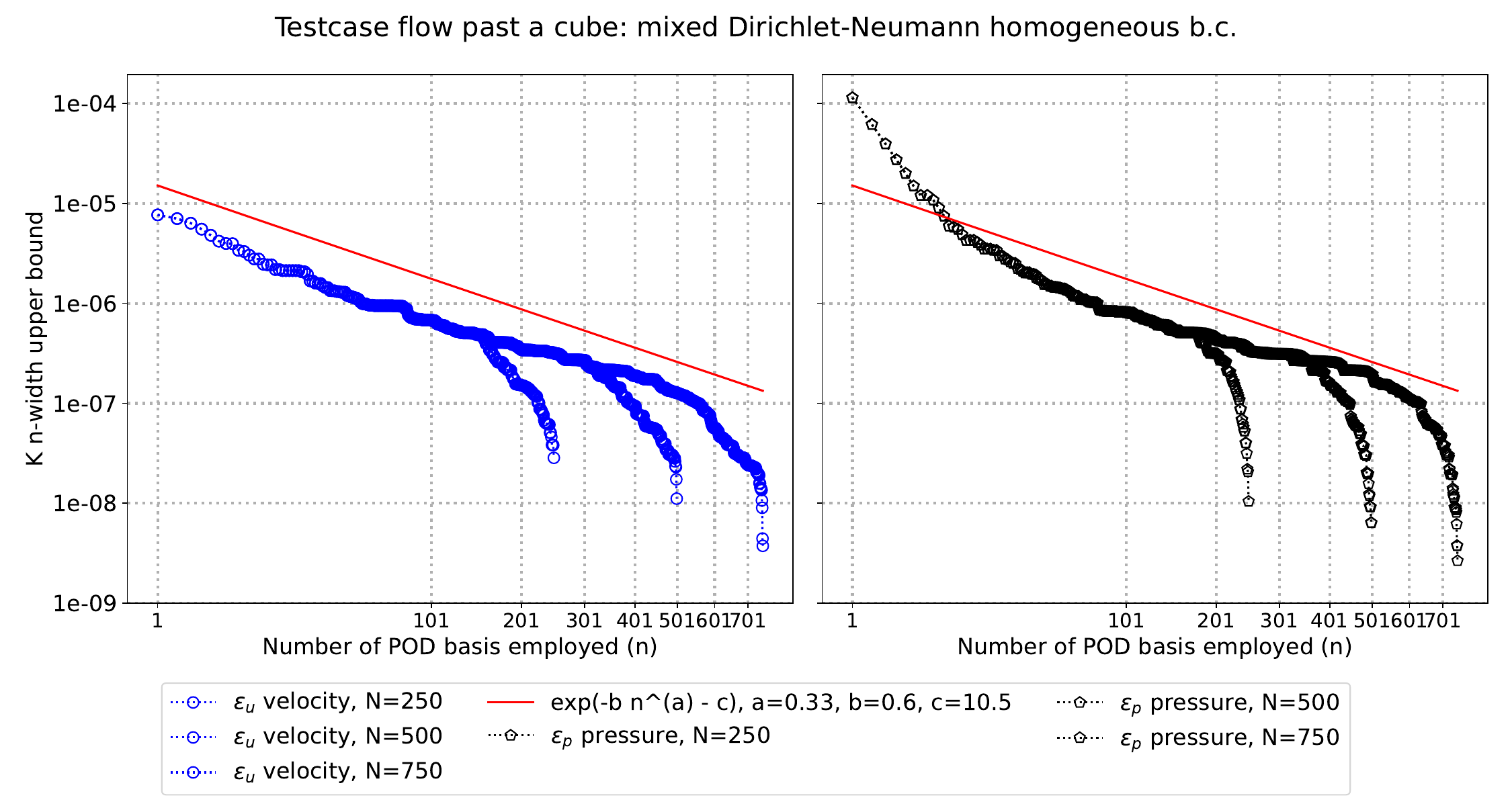}
  \includegraphics[width=1\textwidth]{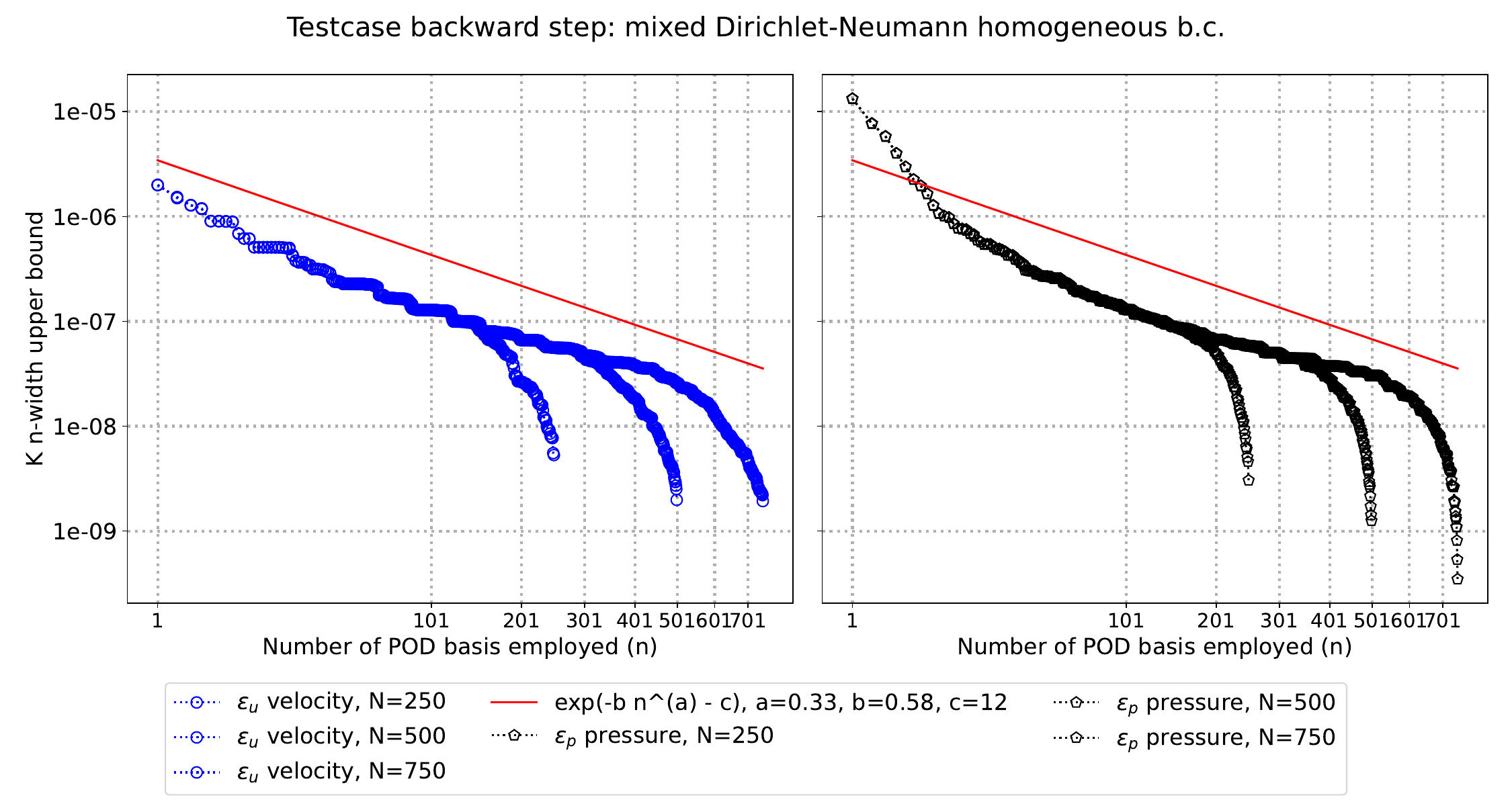}
\caption{Numerically estimated upper bounds $\epsilon_u$ and $\epsilon_p$ of the Kolmogorov $n$-width for mixed Dirichlet-Neumann homogeneous boundary conditions varying the amount of snapshots $N$: flow past a cube (top) and backward facing step (bottom). The asymptotic super-exponential behaviour is due to self-convergence to the high-fidelity numerical reference solution.}
  \label{fig:mixed}
\end{figure}

\subsubsection{Mixed slip and Dirichlet-Neumann homogeneous boundary conditions}
\label{subsec:mixed_slip}

Finally, we consider the more general boundary condition problem, involving both homogeneous Dirichlet-Neumann and slip boundary conditions.
We remark that this setting is still compatible with the assumptions of Theorem~\ref{thm:dnK} and the results of the previous sections.
The discretization and meshes are the same of the previous section and the mixed boundary conditions are imposed, as can be seen from Figure~\ref{fig:domains}, by defining for the two benchmarks, respectively, the Dirichlet boundary as $\Gamma_D = \Gamma_\text{v}$ and $\Gamma_D = \Gamma_\text{c}$, the Neumann boundary as $\Gamma_N = \Gamma_\text{in} \cup \Gamma_\text{out}$ for both geometries, and the slip boundary as $\Gamma_G = \Gamma_\text{w} \cup \Gamma_\text{h}$ and $\Gamma_G = \Gamma_\text{w}$.
The velocity and pressure fields along with their corresponding source term 
$\mathbf{f}_{i, h}\in\mathcal{P}_{N, h}$ for $i=2$ are shown in Figure~\ref{fig:vpbackcube3} 
for the flow past a cube and backward step, top and bottom respectively.

\begin{figure}[tph!]
  \centering
  \includegraphics[width=0.49\textwidth, trim={0 40 0 105}, clip]{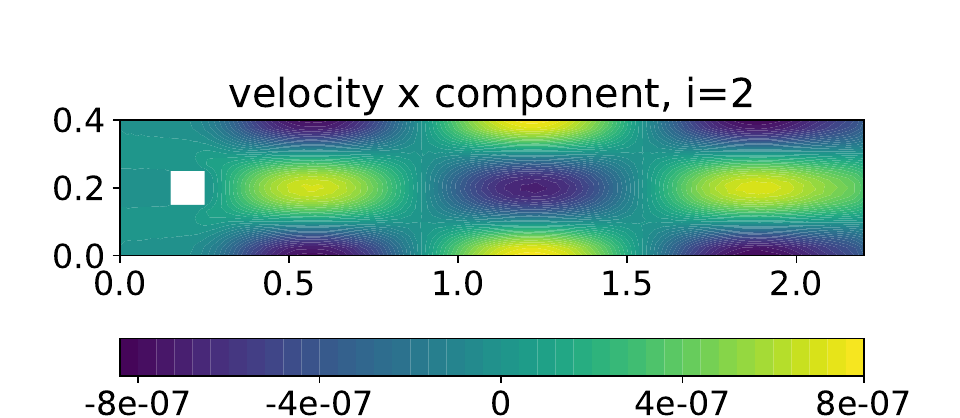}
  \includegraphics[width=0.49\textwidth, trim={0 40 0 105}, clip]{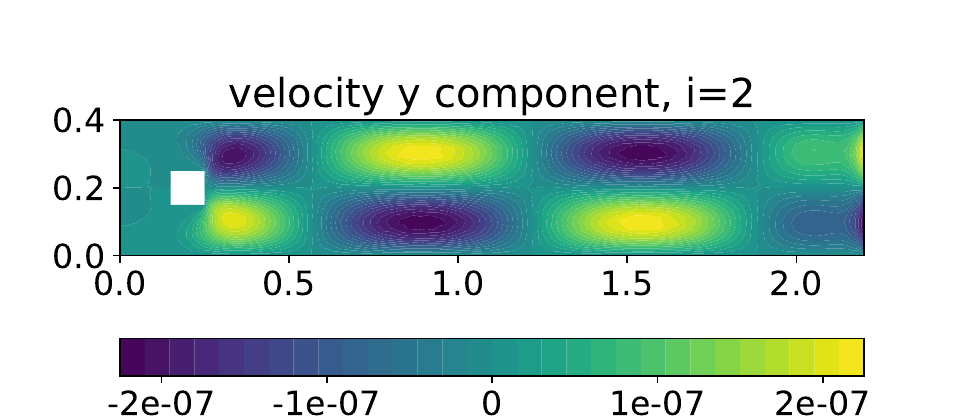}
  \includegraphics[width=0.49\textwidth, trim={0 40 0 105}, clip]{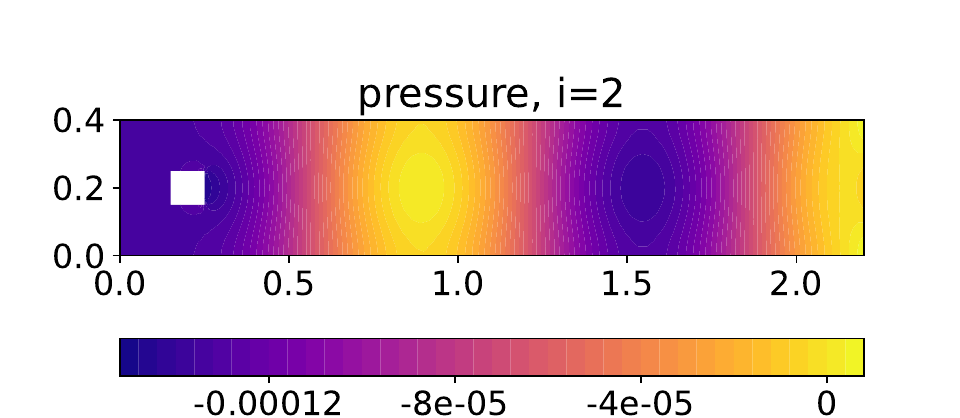}
  \includegraphics[width=0.49\textwidth, trim={0 40 0 105}, clip]{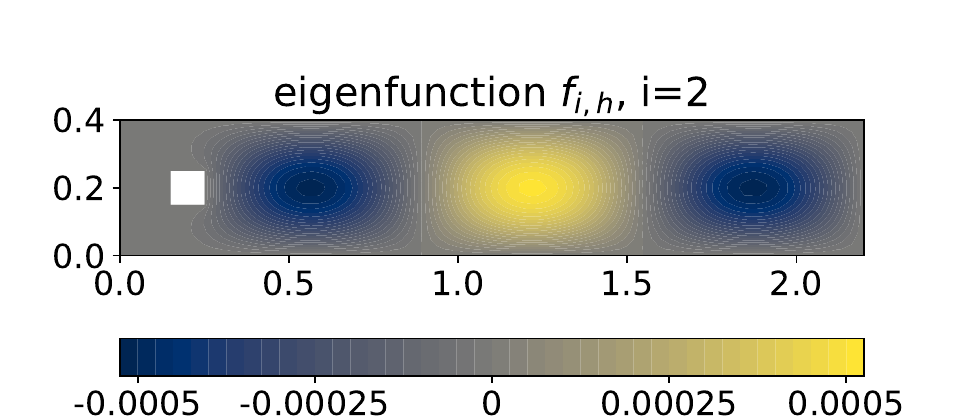}\\[0.5cm]
  \includegraphics[width=0.49\textwidth, trim={0 40 0 105}, clip]{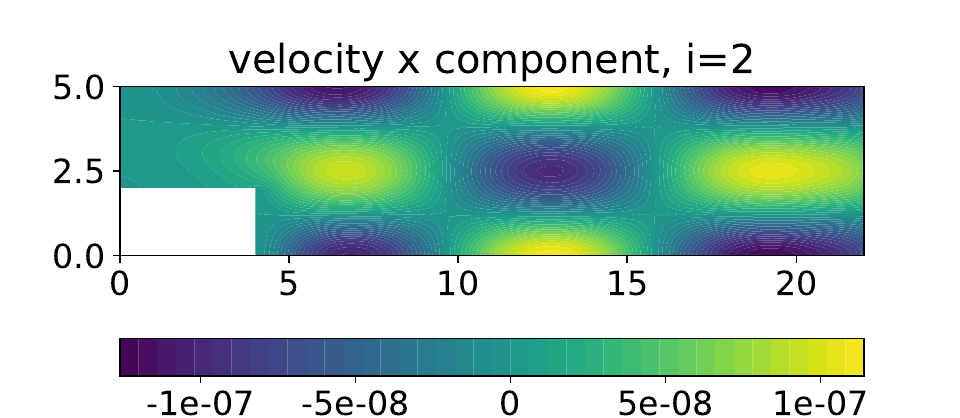}
  \includegraphics[width=0.49\textwidth, trim={0 40 0 105}, clip]{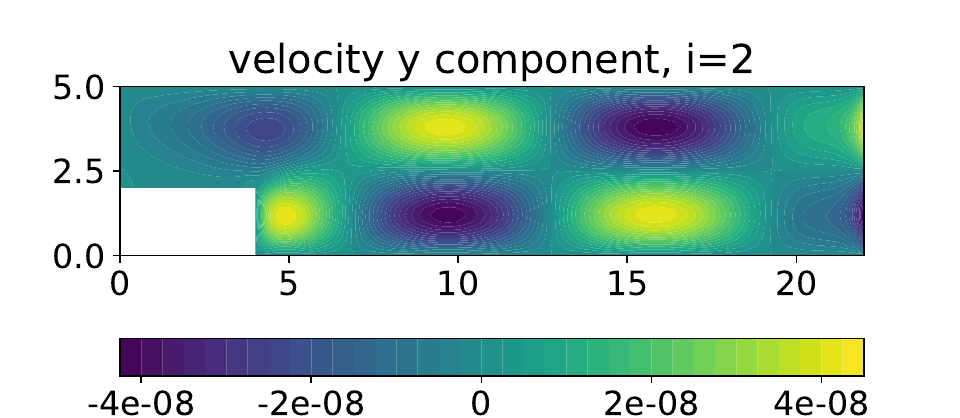}
  \includegraphics[width=0.49\textwidth, trim={0 40 0 105}, clip]{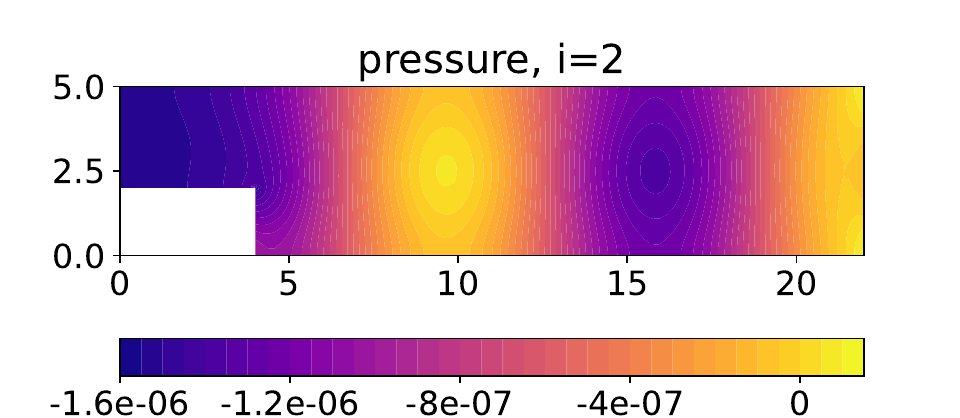}
  \includegraphics[width=0.49\textwidth, trim={0 40 0 105}, clip]{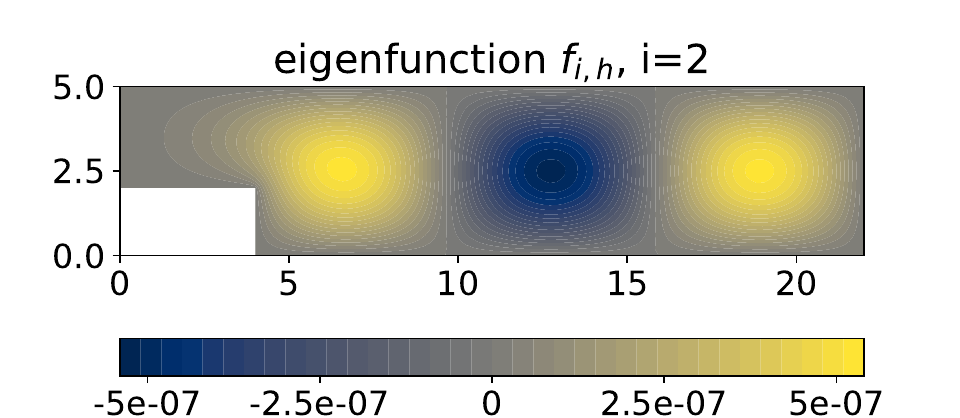}
  \caption{Snapshots of velocity and pressure fields along with corresponding eigenfunction $\mathbf{f}_{i,h}$ of the Dirichlet-Laplace operator, for the INS with mixed slip and homogeneous Dirichlet-Neumann boundary conditions:flow past a cube (top) and backward facing step (bottom).}
  \label{fig:vpbackcube3}
\end{figure}

We report the corresponding decay of the upper bounds on the velocity and pressure reconstruction errors for mixed slip and homogeneous Dirichlet-Neumann boundary conditions reported above. We show that, varying the amount of sampled snapshots $N\in\{250, 500, 750\}$, the KnWs are exhibiting the same exponential decay rate $a = 1/3$ w.r.t.\ the number of POD modes $n$ considered for the reconstruction, as compared to the high-fidelity reference solution, until we reach the self-convergence behaviour.

% The results of the upper bounds $\epsilon_u$ and $\epsilon_p$ of the Kolmogorov $n$-width for homogeneous Dirichlet-Neumann boundary conditions are reported in Figure~\ref{fig:mixed}. We plot the behaviour of the upper bounds $\epsilon_u$ and $\epsilon_p$ while varying the amount of parameters $N\in\{250, 500, 750\}$. Also here,  KnWs are numerically estimated with respect to a numerical, high-fidelity reference solution and the self-convergence behaviour is observed. 

% In this section, we report the decay of the upper bounds on the velocity reconstruction error $\epsilon_u$ (Equation~\eqref{eq:epsilonu}) and pressure reconstruction error $\epsilon_p$ (Equation~\eqref{eq:epsilonp}) for homogeneous Dirichlet-Neumann boundary conditions and slip boundary conditions in Figure~\ref{fig:mixed_slip}. The discretization and meshes are the same of the previous section and the mixed boundary conditions are imposed as can be seen from Figure~\ref{fig:domains_mixed_slip}.

\begin{figure}[tph!]
  \centering
  \includegraphics[width=1\textwidth]{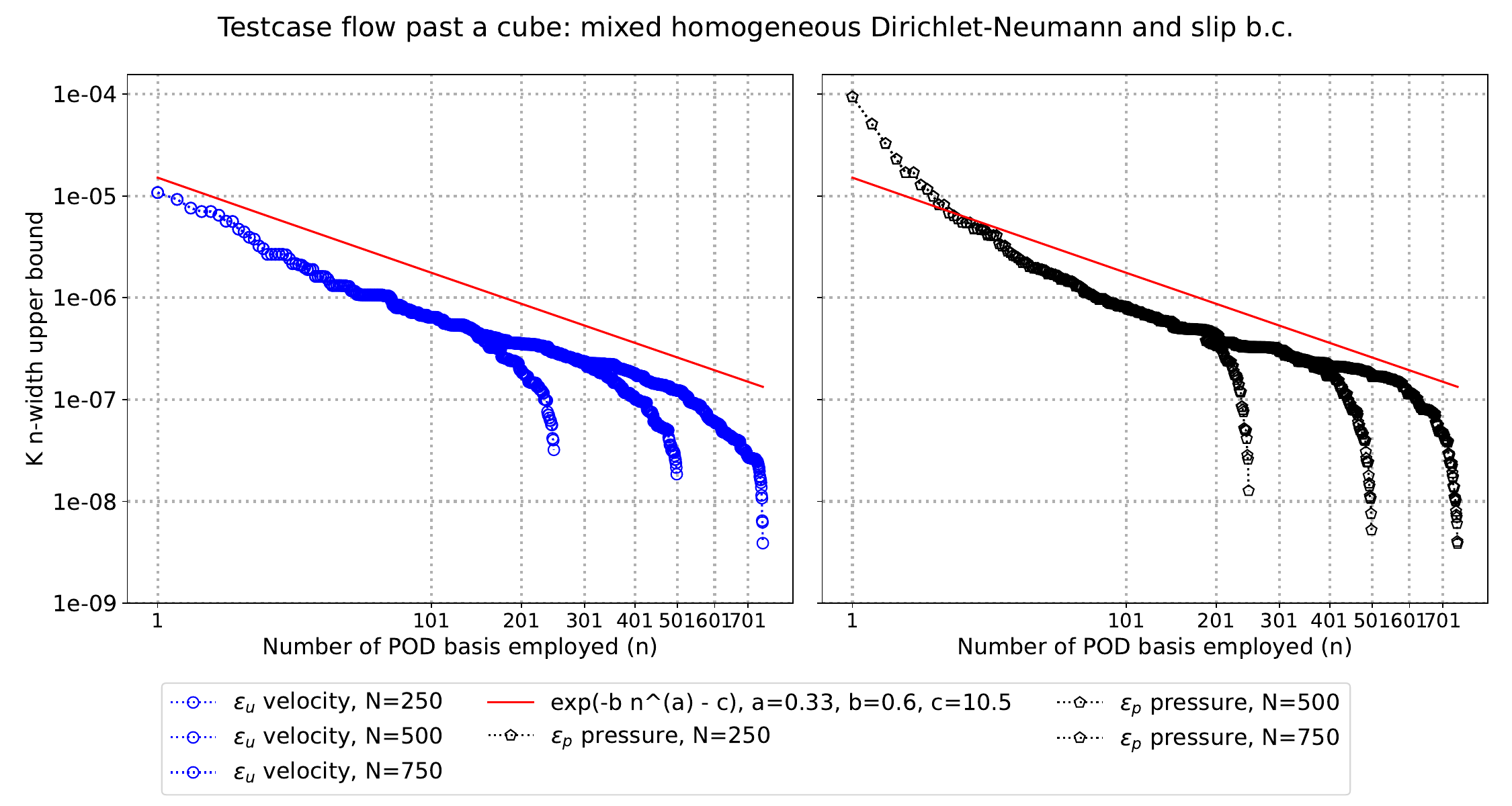}
  \includegraphics[width=1\textwidth]{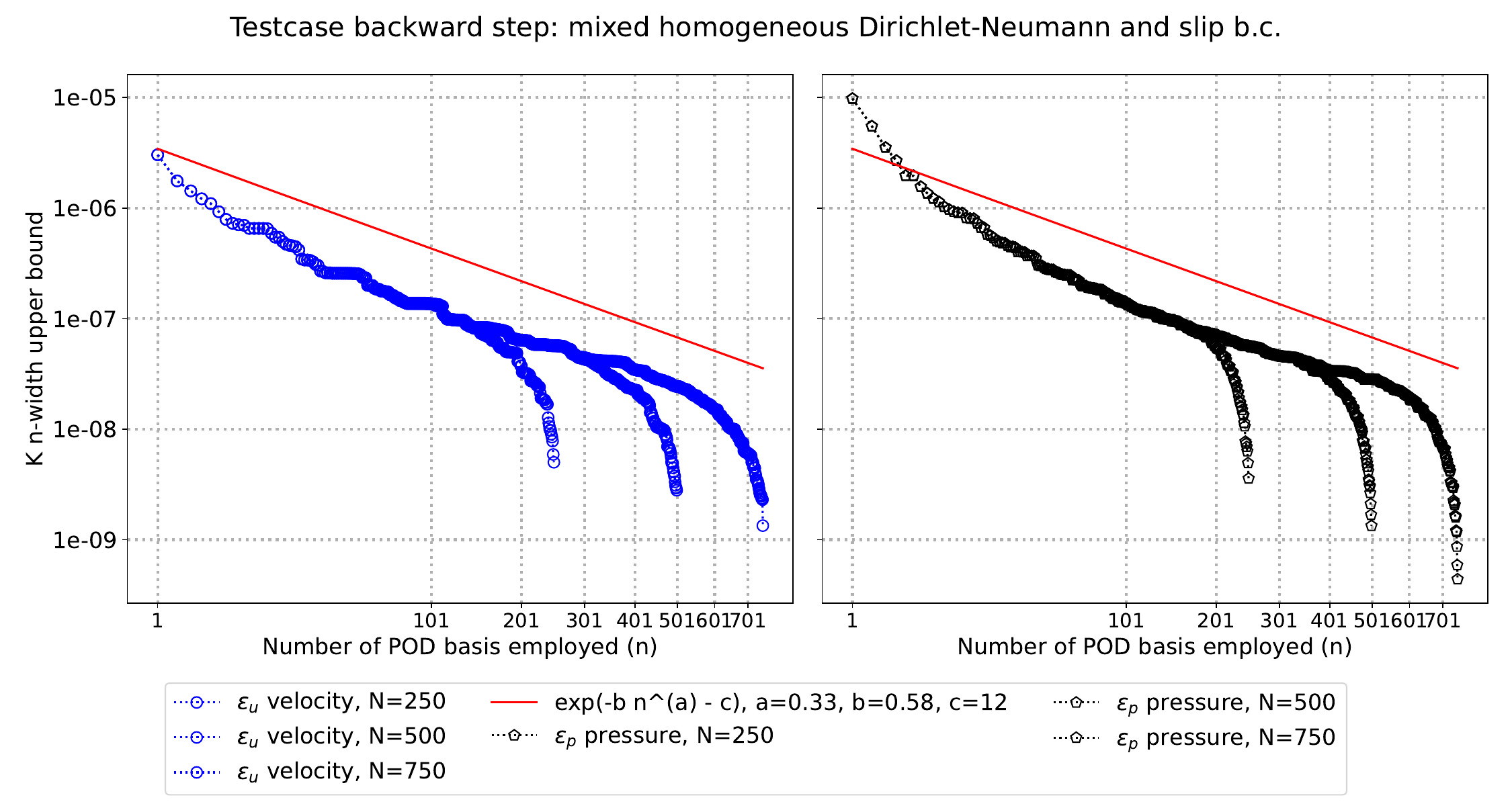}
  \caption{Numerically estimated upper bounds $\epsilon_u$ and $\epsilon_p$ of the Kolmogorov $n$-width for mixed slip and homogeneous Dirichlet-Neumann boundary conditions varying the amount of snapshots $N$: flow past a cube (top) and backward facing step (bottom). The asymptotic super-exponential behaviour is due to self-convergence to the high-fidelity numerical reference solution.}
  % Mixed homogeneous Dirichlet-Neumann and slip boundary conditions Top: flow past a cube in channel. Bottom: backward-facing step. 
  %  Upper bounds $\epsilon_u$ and $\epsilon_p$ of the Kolmogorov $n$-width, varying the number $n$ of POD modes from $1$ to $500$, 
  %  and the number of parameters $N\in\{250, 500, 750\}$. KnWs are numerically estimated with respect to a numerical, high-fidelity reference solution. The asymptotic behaviour is due to self-convergence.
  \label{fig:mixed_slip}
\end{figure}

% \clearpage
% \section{3d backward step}
% \FR{Draft of possible 3d test case. We use hex meshes, with P1-P1 finite element spaces and SUPG-PSPG stabilization. As in the previous section we apply only homogeneous Dirichlet boundary conditions and the source term components are set as the eigenfunctions of the Laplace operator with Dirichlet boundary conditions For high frequency eigenfunctions the solution (too diffusive) does not seem to convergence correctly, we are investigating.}

% \begin{figure}[tph!]
%   \centering
%   \includegraphics[width=0.32\textwidth, trim={200 80 200 180}, clip]{figures/3dbackward/u1.pdf}
%   \includegraphics[width=0.32\textwidth, trim={200 80 200 180}, clip]{figures/3dbackward/u4.pdf}
%   \includegraphics[width=0.32\textwidth, trim={200 80 200 180}, clip]{figures/3dbackward/u7.pdf}\\
%   \includegraphics[width=0.32\textwidth, trim={200 80 200 180}, clip]{figures/3dbackward/p1.pdf}
%   \includegraphics[width=0.32\textwidth, trim={200 80 200 180}, clip]{figures/3dbackward/p4.pdf}
%   \includegraphics[width=0.32\textwidth, trim={200 80 200 180}, clip]{figures/3dbackward/p7.pdf}\\
%   \includegraphics[width=0.32\textwidth, trim={200 80 200 180}, clip]{figures/3dbackward/mesh.pdf}
%   \caption{TODO 3D backward step fix labels. Solutions corresponding to eigenvalues $1, 4, 7$.}
%   \label{fig:vback3D}
% \end{figure}

% \begin{figure}[tph!]
%   \centering
%   \includegraphics[width=1\textwidth]{figures/3dbackward/3dbackward_rec_errors.pdf}
%   \caption{TODO.}
%   \label{fig:backward3d}
% \end{figure}

\section{Conclusions and Perspectives}
\label{sec:conclusions}
In this work, we have addressed the model order reduction of the stationary, viscous, incompressible Navier-Stokes equations in a polygon $\Omega\subset \R^2$, subject to Assumption \ref{ass:Dirich} and to mixed boundary conditions and corner-weighted analytic in $\overline{\Omega}$ volume forcing. Leveraging analytic regularity results and a small data hypothesis, we established exponential bounds on the Kolmogorov $n$-widths of the solution sets in $H^1(\Omega)^2\times L^2(\Omega)$ (Theorem~\ref{thm:dnK}). 
The key idea of the proof was to use recent analytic regularity results in scales of corner-weighted Sobolev spaces of Kondrat'ev type, combined with exponential convergence rate bounds of suitable, so-called ``$hp$-Finite Element'' approximations of velocity and pressure. 
These results were used as comparisons for the optimal subspaces of dimension $n$, confirming that 
the solution manifold can be approximated with a convergence rate of the form $C \exp(-bn^{1/3})$.
This is highly relevant for reduced basis methodologies.

Our numerical experiments support the theoretical predictions, verifying the sharpness of these Kolmogorov $n$-widths of solution sets, for different geometries and boundary conditions. The reconstruction errors, evaluated for a finite-dimensional parameter space made of the eigenfunctions of the Laplace operator, exhibit the predicted exponential convergence rates with respect to the number of reduced modes, for mixed Dirichlet, Neumann and no-slip boundary conditions.
Here, convergence tests were performed w.r.t.\ a high-fidelity solution with a divergence-stable, low order velocity-pressure pair. 

The findings presented in this manuscript pave the way for further research in various directions, here we comment on some implications, and also indicate generalizations. 
While the basic, analytic regularity result Theorem~\ref{th:analytic-u}, holds only under the assumption of small data, we believe that corresponding analytic regularity, and exponential rates of $hp$-approximation remains valid also for moderately larger values of the Reynolds parameter, \emph{along regular branches of isolated solutions} in the sense of \cite{BRR}.
Furthermore, the regularity theory in \cite{HMS21_971} is formulated only for Newtonian fluids, i.e.\ for a linear stress-strain relation.
We expect similar results also for nonlinear, strongly monotone, analytic constitutive relations, again under a small data hypothesis.
Moreover, future work could explore extensions to three-dimensional computational domains, curved boundaries, and alternative boundary conditions. Our results might also benefit the development of machine learning surrogates combining linear compression strategies, such as POD, to perform dimensionality reduction and neural networks to learn the evolution of the latent coefficients.
% \todo{[I am not certain what means ``linear ansatz spaces''; can we discuss this?]}

%\appendix
%\setcounter{equation}{0}
%\renewcommand\theequation{\arabic{equation}}
%\input{./sections/appendix.tex}

\section*{Acknowledgements}
\noindent \textbf{FP} and \textbf{GR} acknowledge the support provided by the European Union - NextGenerationEU, in the framework of the iNEST - Interconnected Nord-Est Innovation Ecosystem (iNEST ECS00000043 - CUP G93C22000610007) consortium and its CC5 Young Researchers initiative.
The authors also like to acknowledge INdAM-GNCS for its support.

\bibliographystyle{abbrv}
\bibliography{biblio}

@article{cohen2015approximation,
  title={Approximation of high-dimensional parametric {PDEs}},
  author={Cohen, Albert and DeVore, Ronald},
  journal={Acta Numerica},
  volume={24},
  pages={1--159},
  year={2015},
  publisher={Cambridge University Press}
}

@article{HMS21_971,
   author = {Yanchen He and Carlo Marcati and Christoph Schwab},
    title = {Analytic regularity of solutions to the {Navier-Stokes} equations with mixed boundary conditions in polygons},
    journal = {SIAM Journ. Math. Analysis},
    volume = {56},
    number = {2},
    pages = {2488--2520},
    doi = {https://doi.org/10.1137/22M1527428},
    year = {2024}
}

@article {Babuska1988Reg,
    AUTHOR = {Babu\v{s}ka, I. and Guo, B. Q.},
     TITLE = {Regularity of the solution of elliptic problems with piecewise
              analytic data. {I}. {B}oundary value problems for linear
              elliptic equation of second order},
   JOURNAL = {SIAM J. Math. Anal.},
  FJOURNAL = {SIAM Journal on Mathematical Analysis},
    VOLUME = {19},
      YEAR = {1988},
    NUMBER = {1},
     PAGES = {172--203},
      ISSN = {0036-1410},
   MRCLASS = {35B65 (35J25 65N99)},
  MRNUMBER = {924554},
MRREVIEWER = {Neil M. Wigley},
       DOI = {10.1137/0519014},
       URL = {https://doi.org/10.1137/0519014},
}

@book {Quarteroni2016,
    AUTHOR = {Quarteroni, Alfio and Manzoni, Andrea and Negri, Federico},
     TITLE = {Reduced basis methods for partial differential equations},
    SERIES = {Unitext},
    VOLUME = {92},
 PUBLISHER = {Springer, Cham},
      YEAR = {2016},
     PAGES = {xi+296},
      ISBN = {978-3-319-15430-5; 978-3-319-15431-2},
   MRCLASS = {65-02 (35A35)},
  MRNUMBER = {3379913},
MRREVIEWER = {Gianluigi Rozza},
       noshow_DOI = {10.1007/978-3-319-15431-2},
       noshow_URL = {https://doi.org/10.1007/978-3-319-15431-2},
}

@book {Hesthaven2016,
    AUTHOR = {Hesthaven, Jan S. and Rozza, Gianluigi and Stamm, Benjamin},
     TITLE = {Certified reduced basis methods for parametrized partial differential equations},
    SERIES = {SpringerBriefs in Mathematics},
 PUBLISHER = {Springer, Cham; BCAM Basque Center for Applied Mathematics,
              Bilbao},
      YEAR = {2016},
     PAGES = {xiii+131},
      ISBN = {978-3-319-22469-5; 978-3-319-22470-1},
   MRCLASS = {65-02 (35J25 35K20)},
  MRNUMBER = {3408061},
MRREVIEWER = {Mario Ohlberger},
       noshow_DOI = {10.1007/978-3-319-22470-1},
       noshow_URL = {https://doi.org/10.1007/978-3-319-22470-1},
}

@book{BennerSystemDataDrivenMethods2021,
  title = {System- and {{Data-Driven Methods}} and {{Algorithms}}},
  author = {Benner, Peter and Grivet Talocia, Sefano and Quarteroni, Alfio and Rozza, Gianluigi and Schilders, Wil and Silveira, Luis Miguel},
  year = {2021},
  journal = {Volume 1 System- and Data-Driven Methods and Algorithms},
  volume = {1},
  publisher = {De Gruyter},
  doi = {10.1515/9783110498967},
  urldate = {2022-12-08},
  isbn = {978-3-11-049896-7}
}

@book{BennerSnapshotBasedMethodsAlgorithms2020,
  title = {Snapshot-{{Based Methods}} and {{Algorithms}}},
  author = {Benner, Peter and Grivet Talocia, Sefano and Quarteroni, Alfio and Rozza, Gianluigi and Schilders, Wil and Silveira, Luis Miguel},
  year = {2020},
  journal = {Volume 2 Snapshot-Based Methods and Algorithms},
  volume = {2},
  publisher = {De Gruyter},
  doi = {10.1515/9783110671490},
  urldate = {2022-12-11},
  isbn = {978-3-11-067149-0}
}

@article {Veroy2005,
    AUTHOR = {Veroy, K. and Patera, A. T.},
     TITLE = {Certified real-time solution of the parametrized steady
              incompressible {N}avier-{S}tokes equations: rigorous
              reduced-basis a posteriori error bounds},
   JOURNAL = {Internat. J. Numer. Methods Fluids},
  FJOURNAL = {International Journal for Numerical Methods in Fluids},
    VOLUME = {47},
      YEAR = {2005},
    NUMBER = {8-9},
     PAGES = {773--788},
      ISSN = {0271-2091},
   MRCLASS = {76D05 (65N15 76M30)},
  MRNUMBER = {2123791},
       noshow_DOI = {10.1002/fld.867},
       noshow_URL = {https://doi.org/10.1002/fld.867},
}

@article{CSMSuri99,
    AUTHOR = {Schwab, C. and Suri, M.},
     TITLE = {Mixed {$hp$} finite element methods for {S}tokes and non-{N}ewtonian flow},
   JOURNAL = {Comput. Methods Appl. Mech. Engrg.},
  FJOURNAL = {Computer Methods in Applied Mechanics and Engineering},
    VOLUME = {175},
      YEAR = {1999},
    NUMBER = {3-4},
     PAGES = {217--241},
      ISSN = {0045-7825},
   MRCLASS = {76M10 (65M60 76A05 76D07)},
  MRNUMBER = {1702217},
MRREVIEWER = {Gert Lube},
       NOSHOW_DOI = {10.1016/S0045-7825(98)00355-7},
       NOSHOW_URL = {https://NOSHOW_DOI.org/10.1016/S0045-7825(98)00355-7},
}

@incollection{OrltSandig95,
    AUTHOR = {Orlt, Matthias and S\"{a}ndig, Anna-Margarete},
     TITLE = {Regularity of viscous {N}avier-{S}tokes flows in nonsmooth domains},
 BOOKTITLE = {Boundary value problems and integral equations in nonsmooth
              domains ({L}uminy, 1993)},
    SERIES = {Lecture Notes in Pure and Appl. Math.},
    VOLUME = {167},
     PAGES = {185--201},
 PUBLISHER = {Dekker, New York},
      YEAR = {1995},
   MRCLASS = {76D05 (35Q30 65P05)},
  MRNUMBER = {1301349},
MRREVIEWER = {B. Kellogg},
}

@unpublished{nonlinear,
author = {Maday, Yvon and Marcati, Carlo},
title = {Regularity in weighted {S}obolev spaces and analysis of the $hp$
discontinuous {G}alerkin approximation of elliptic nonlinear eigenvalue problems with point singularities},
year = {2018}
}

@article{Guo2006a,
author = {Guo, Benqi Qui and Schwab, Christoph},
NOSHOW_DOI = {10.1016/J.CAM.2005.02.018},
issn = {0377-0427},
journal = {J. Comput. Appl. Math.},
fjournal = {Journal of Computational and Applied Mathematics},
number = {1-2},
pages = {487--519},
publisher = {North-Holland},
title = {{Analytic regularity of Stokes flow on polygonal domains in countably weighted Sobolev spaces}},
NOSHOW_url = {https://www.sciencedirect.com/science/article/pii/S0377042705002487},
volume = {190},
year = {2006}
}

@article{linear,
title ={{Regularity and $hp$ discontinuous Galerkin finite element approximation of linear
  elliptic eigenvalue problems with singular potentials}},
author = {Maday, Yvon and Marcati, Carlo},
year = {2018},
note = {submitted}
}

@book{Mazya2010,
address = {Providence, Rhode Island},
author = {Maz'ya, Vladimir and Rossmann, J{\"{u}}rgen},
NOSHOW_DOI = {10.1090/surv/162},
isbn = {9780821849835},
publisher = {American Mathematical Society},
series = {Mathematical Surveys and Monographs},
title = {{Elliptic Equations in Polyhedral Domains}},
NOSHOW_url = {http://www.ams.org/surv/162},
volume = {162},
year = {2010}
}

@article{Babuska1988,
    AUTHOR = {Babu\v{s}ka, I. and Guo, B. Q.},
     TITLE = {The {$h$}-{$p$} version of the finite element method for domains with curved boundaries},
   JOURNAL = {SIAM J. Numer. Anal.},
  FJOURNAL = {SIAM Journal on Numerical Analysis},
    VOLUME = {25},
      YEAR = {1988},
    NUMBER = {4},
     PAGES = {837--861},
      ISSN = {0036-1429},
   MRCLASS = {65N30},
  MRNUMBER = {954788},
       NOSHOW_DOI = {10.1137/0725048},
}

@article{cohen2011analytic,
  title={Analytic regularity and polynomial approximation of parametric and stochastic elliptic PDE's},
  author={Cohen, Albert and Devore, Ronald and Schwab, Christoph},
  journal={Analysis and Applications},
  volume={9},
  number={01},
  pages={11--47},
  year={2011},
  publisher={World Scientific}
}

@article{Melenk2000,
    AUTHOR = {Melenk, J. M.},
     TITLE = {On {$n$}-widths for elliptic problems},
   JOURNAL = {J. Math. Anal. Appl.},
  FJOURNAL = {Journal of Mathematical Analysis and Applications},
    VOLUME = {247},
      YEAR = {2000},
    NUMBER = {1},
     PAGES = {272--289},
      ISSN = {0022-247X},
   MRCLASS = {35J99 (35B99 41A46 46N20)},
  MRNUMBER = {1766938},
MRREVIEWER = {Pekka Neittaanm\"{a}ki},
       DOI = {10.1006/jmaa.2000.6862},
       URL = {https://doi.org/10.1006/jmaa.2000.6862},
}

@book{FMRT2001,
    AUTHOR = {Foias, C. and Manley, O. and Rosa, R. and Temam, R.},
     TITLE = {Navier-{S}tokes equations and turbulence},
    SERIES = {Encyclopedia of Mathematics and its Applications},
    VOLUME = {83},
 PUBLISHER = {Cambridge University Press, Cambridge},
      YEAR = {2001},
     PAGES = {xiv+347},
      ISBN = {0-521-36032-3},
   MRCLASS = {76-02 (35Q30 37L30 37N10 76D05 76D06 76F05 76F20)},
  MRNUMBER = {1855030},
MRREVIEWER = {Xiaoming Wang},
       DOI = {10.1017/CBO9780511546754},
       URL = {https://doi.org/10.1017/CBO9780511546754},
}

@article{StabRozzaPODG2018,
    AUTHOR = {Stabile, Giovanni and Rozza, Gianluigi},
     TITLE = {Finite volume {POD}-{G}alerkin stabilised reduced order
              methods for the parametrised incompressible {N}avier-{S}tokes
              equations},
   JOURNAL = {Comput. \& Fluids},
  FJOURNAL = {Computers \& Fluids. An International Journal},
    VOLUME = {173},
      YEAR = {2018},
     PAGES = {273--284},
      ISSN = {0045-7930},
   MRCLASS = {65M08 (76D05 76M12)},
  MRNUMBER = {3843669},
       DOI = {10.1016/j.compfluid.2018.01.035},
       URL = {https://doi.org/10.1016/j.compfluid.2018.01.035},
}

@article{SirovichTurbulenceDynamicsCoherent1987a,
  title = {Turbulence and the {{Dynamics}} of {{Coherent Structures Part I}}: {{Coherent Structures}}},
  shorttitle = {Turbulence and the {{Dynamics}} of {{Coherent Structures Part I}}},
  author = {Sirovich, Lawrence},
  year = {1987},
  journal = {Quarterly of Applied Mathematics},
  volume = {45},
  number = {3},
  eprint = {43637457},
  eprinttype = {jstor},
  pages = {561--571},
  publisher = {Brown University},
  urldate = {2024-10-08}
}

@article{RozzaStabilityReducedBasis2007,
  title = {On the Stability of the Reduced Basis Method for {{Stokes}} Equations in Parametrized Domains},
  author = {Rozza, Gianluigi and Veroy, Karen},
  year = {2007},
  journal = {Computer Methods in Applied Mechanics and Engineering},
  volume = {196},
  number = {7},
  pages = {1244--1260},
  doi = {10.1016/j.cma.2006.09.005},
  urldate = {2022-12-11}
}

@article{BarraultEmpiricalInterpolationMethod2004,
  title = {An `Empirical Interpolation' Method: Application to Efficient Reduced-Basis Discretization of Partial Differential Equations},
  shorttitle = {An `Empirical Interpolation' Method},
  author = {Barrault, Maxime and Maday, Yvon and Nguyen, Ngoc Cuong and Patera, Anthony T.},
  year = {2004},
  journal = {Comptes Rendus Mathematique},
  volume = {339},
  number = {9},
  pages = {667--672},
  doi = {10.1016/j.crma.2004.08.006},
  urldate = {2022-08-01}
}

@article{ChaturantabutNonlinearModelReduction2010,
  title = {Nonlinear {{Model Reduction}} via {{Discrete Empirical Interpolation}}},
  author = {Chaturantabut, Saifon and Sorensen, Danny C.},
  year = {2010},
  journal = {SIAM Journal on Scientific Computing},
  volume = {32},
  number = {5},
  pages = {2737--2764},
  publisher = {{Society for Industrial and Applied Mathematics}},
  doi = {10.1137/090766498},
  urldate = {2022-08-01}
}

@article{HerreroRBReducedBasis2013,
  title = {{{RB}} ({{Reduced}} Basis) for {{RB}} ({{Rayleigh}}--{{B{\'e}nard}})},
  author = {Herrero, H. and Maday, Y. and Pla, F.},
  year = {2013},
  journal = {Computer Methods in Applied Mechanics and Engineering},
  volume = {261--262},
  pages = {132--141},
  doi = {10.1016/j.cma.2013.02.018},
  urldate = {2023-12-19}
}

@article{PichiArtificialNeuralNetwork2021a,
  title = {An Artificial Neural Network Approach to Bifurcating Phenomena in Computational Fluid Dynamics},
  author = {Pichi, Federico and Ballarin, Francesco and Rozza, Gianluigi and Hesthaven, Jan S.},
  year = {2023},
  journal = {Computers \& Fluids},
  volume = {254},
  pages = {105813},
  doi = {10.1016/j.compfluid.2023.105813}
}

@article{DengLoworderModelSuccessive2020a,
  title = {Low-Order Model for Successive Bifurcations of the Fluidic Pinball},
  author = {Deng, Nan and Noack, Bernd R. and Morzy{\'n}ski, Marek and Pastur, Luc R.},
  year = {2020},
  journal = {Journal of Fluid Mechanics},
  volume = {884},
  pages = {A37},
  doi = {10.1017/jfm.2019.959},
  urldate = {2023-06-23}
}

@article{StabileEfficientGeometricalParametrization2020,
  title = {Efficient Geometrical Parametrization for Finite-Volume-Based Reduced Order Methods},
  author = {Stabile, Giovanni and Zancanaro, Matteo and Rozza, Gianluigi},
  year = {2020},
  journal = {International Journal for Numerical Methods in Engineering},
  volume = {121},
  number = {12},
  pages = {2655--2682},
  doi = {10.1002/nme.6324},
  urldate = {2022-12-13}
}

@article{BruntonMachineLearningFluid2020,
  title = {Machine {{Learning}} for {{Fluid Mechanics}}},
  author = {Brunton, Steven L. and Noack, Bernd R. and Koumoutsakos, Petros},
  year = {2020},
  journal = {Annual Review of Fluid Mechanics},
  volume = {52},
  number = {Volume 52, 2020},
  pages = {477--508},
  publisher = {Annual Reviews},
  doi = {10.1146/annurev-fluid-010719-060214},
  urldate = {2024-04-17}
}

@article{VinuesaEnhancingComputationalFluid2022,
  title = {Enhancing Computational Fluid Dynamics with Machine Learning},
  author = {Vinuesa, Ricardo and Brunton, Steven L.},
  year = {2022},
  journal = {Nature Computational Science},
  volume = {2},
  number = {6},
  pages = {358--366},
  publisher = {Nature Publishing Group},
  doi = {10.1038/s43588-022-00264-7},
  urldate = {2024-01-05}
}

@misc{QuainiBridgingLargeEddy2024,
  title = {Bridging {{Large Eddy Simulation}} and {{Reduced Order Modeling}} of {{Convection-Dominated Flows}} through {{Spatial Filtering}}: {{Review}} and {{Perspectives}}},
  shorttitle = {Bridging {{Large Eddy Simulation}} and {{Reduced Order Modeling}} of {{Convection-Dominated Flows}} through {{Spatial Filtering}}},
  author = {Quaini, Annalisa and San, Omer and Veneziani, Alessandro and Iliescu, Traian},
  year = {2024},
  number = {arXiv:2407.00231},
  eprint = {2407.00231},
  primaryclass = {physics},
  publisher = {arXiv},
  urldate = {2024-07-03},
  archiveprefix = {arXiv}
}

@article{RomorNonlinearManifoldReducedOrder2023,
  title = {Non-Linear {{Manifold Reduced-Order Models}} with {{Convolutional Autoencoders}} and {{Reduced Over-Collocation Method}}},
  author = {Romor, Francesco and Stabile, Giovanni and Rozza, Gianluigi},
  year = {2023},
  journal = {Journal of Scientific Computing},
  volume = {94},
  number = {3},
  pages = {74},
  doi = {10.1007/s10915-023-02128-2},
  urldate = {2023-02-24}
}

@article{LeeModelReductionDynamical2020,
  title = {Model Reduction of Dynamical Systems on Nonlinear Manifolds Using Deep Convolutional Autoencoders},
  author = {Lee, Kookjin and Carlberg, Kevin T.},
  year = {2020},
  journal = {Journal of Computational Physics},
  volume = {404},
  pages = {108973},
  doi = {10.1016/j.jcp.2019.108973},
  urldate = {2022-11-27}
}

@article{PichiGraphConvolutionalAutoencoder2024,
  title = {A Graph Convolutional Autoencoder Approach to Model Order Reduction for Parametrized {{PDEs}}},
  author = {Pichi, Federico and Moya, Beatriz and Hesthaven, Jan S.},
  year = {2024},
  journal = {Journal of Computational Physics},
  volume = {501},
  pages = {112762},
  doi = {10.1016/j.jcp.2024.112762},
  urldate = {2024-01-18}
}

@article{MaulikReducedorderModelingAdvectiondominated2021,
  title = {Reduced-Order Modeling of Advection-Dominated Systems with Recurrent Neural Networks and Convolutional Autoencoders},
  author = {Maulik, Romit and Lusch, Bethany and Balaprakash, Prasanna},
  year = {2021},
  journal = {Physics of Fluids},
  volume = {33},
  number = {3},
  pages = {037106},
  doi = {10.1063/5.0039986},
  urldate = {2024-10-08}
}

@article{DSThWhpDGStok,
    AUTHOR = {Sch\"{o}tzau, Dominik and Wihler, Thomas P.},
     TITLE = {Exponential convergence of mixed hp-{DGFEM} for {S}tokes flow in polygons},
   JOURNAL = {Numer. Math.},
  FJOURNAL = {Numerische Mathematik},
    VOLUME = {96},
      YEAR = {2003},
    NUMBER = {2},
     PAGES = {339--361},
      ISSN = {0029-599X},
   MRCLASS = {65N30 (65N12 76D07 76M10)},
  MRNUMBER = {2021494},
       DOI = {10.1007/s00211-003-0478-5},
       URL = {https://doi.org/10.1007/s00211-003-0478-5},
}

@Article{FS20_2675,
    author = {M. Feischl and Christoph Schwab},
    title = {{Exponential convergence in $H^1$ of $hp$-FEM for Gevrey regularity with isotropic singularities}},
    journal = {Numer. Math.},
    volume = {144},
    number = {2},
    pages = {323--346},
    doi = {http://doi.org/10.1007/s00211-019-01085-z},
    year = {2020}
}

@book{PinkusBook,
    AUTHOR = {Pinkus, Allan},
     TITLE = {{$n$}-widths in approximation theory},
    SERIES = {Ergebnisse der Mathematik und ihrer Grenzgebiete (3) [Results
              in Mathematics and Related Areas (3)]},
    VOLUME = {7},
 PUBLISHER = {Springer-Verlag, Berlin},
      YEAR = {1985},
     PAGES = {x+291},
      ISBN = {3-540-13638-X},
   MRCLASS = {41-02 (41A46 46E35)},
  MRNUMBER = {774404},
MRREVIEWER = {Klaus H\"{o}llig},
       DOI = {10.1007/978-3-642-69894-1},
       URL = {https://doi.org/10.1007/978-3-642-69894-1},
}

@article{Ohlberger2016Reduced,
  title={Reduced Basis Methods: Success, Limitations and Future Challenges},
  author={Mario Ohlberger and Stephan Rave},
  journal={Proceedings of ALGORITMY 2016},
  pages={1--12},
  year={2016},
  eprint={1511.02021},
  archivePrefix={arXiv},
  primaryClass={math.NA},
  doi={10.48550/arXiv.1511.02021},
  url={https://arxiv.org/abs/1511.02021}
}

@article{Arbes2023Kolmogorov,
  title = {The {{Kolmogorov N-width}} for Linear Transport: Exact Representation and the Influence of the Data},
  shorttitle = {The {{Kolmogorov N-width}} for Linear Transport},
  author = {Arbes, Florian and Greif, Constantin and Urban, Karsten},
  year = {2025},
  journal = {Advances in Computational Mathematics},
  volume = {51},
  number = {2},
  pages = {13},
  doi = {10.1007/s10444-025-10224-0}
}

@article{Moler1968Bounds,
  title={Bounds for Eigenvalues and Eigenvectors of Symmetric Operators},
  author={C. B. Moler and L. E. Payne},
  journal={SIAM Journal on Numerical Analysis},
  volume={5},
  number={1},
  pages={64--70},
  year={1968},
  doi={10.1137/0705004},
  url={https://doi.org/10.1137/0705004}
}

@article{Grebenkov2013Geometrical,
  title={Geometrical Structure of Laplacian Eigenfunctions},
  author={Denis S. Grebenkov and Binh-Thanh Nguyen},
  journal={SIAM Review},
  volume={55},
  number={4},
  pages={601--667},
  year={2013},
  doi={10.1137/120880173},
  url={https://doi.org/10.1137/120880173}
}

@Inbook{Rozza2024,
author="Rozza, Gianluigi
and Ballarin, Francesco
and Scandurra, Leonardo
and Pichi, Federico",
title="Worked Out Problem 12: Navier-Stokes System for a Backward-Facing Step",
bookTitle="Real Time Reduced Order Computational Mechanics: Parametric PDEs Worked Out Problems",
year="2024",
publisher="Springer Nature Switzerland",
address="Cham",
pages="119--127",
isbn="978-3-031-49892-3",
doi="10.1007/978-3-031-49892-3_13",
url="https://doi.org/10.1007/978-3-031-49892-3_13"
}

@article{BRR,
    AUTHOR = {Brezzi, F. and Rappaz, J. and Raviart, P.-A.},
     TITLE = {Finite-dimensional approximation of nonlinear problems. {I}.
              {B}ranches of nonsingular solutions},
   JOURNAL = {Numer. Math.},
  FJOURNAL = {Numerische Mathematik},
    VOLUME = {36},
      YEAR = {1980/81},
    NUMBER = {1},
     PAGES = {1--25},
      ISSN = {0029-599X},
   MRCLASS = {65J15 (58E07 65N30)},
  MRNUMBER = {595803},
MRREVIEWER = {Erich Bohl},
       DOI = {10.1007/BF01395985},
       URL = {https://doi.org/10.1007/BF01395985},
}

@article{Couplet2005,
  title = {Calibrated reduced-order POD-Galerkin system for fluid flow modelling},
  volume = {207},
  ISSN = {0021-9991},
  url = {http://dx.doi.org/10.1016/j.jcp.2005.01.008},
  DOI = {10.1016/j.jcp.2005.01.008},
  number = {1},
  journal = {Journal of Computational Physics},
  publisher = {Elsevier BV},
  author = {Couplet,  M. and Basdevant,  C. and Sagaut,  P.},
  year = {2005},
  month = jul,
  pages = {192–220}
}

@article{Canuto2009,
author = {Canuto, Claudio and Tonn, Timo and Urban, Karsten},
title = {A Posteriori Error Analysis of the Reduced Basis Method for Nonaffine Parametrized Nonlinear PDEs},
journal = {SIAM Journal on Numerical Analysis},
volume = {47},
number = {3},
pages = {2001-2022},
year = {2009},
doi = {10.1137/080724812},

URL = { 
    
        https://doi.org/10.1137/080724812
    
    

},
eprint = { 
    
        https://doi.org/10.1137/080724812
    
    

}
}

@article{Romor2023Explicable,
  title     = {Explicable Hyper-Reduced Order Models on Nonlinearly Approximated Solution Manifolds of Compressible and Incompressible Navier-Stokes Equations},
  author    = {Francesco Romor and Giovanni Stabile and Gianluigi Rozza},
  journal   = {Journal of Computational Physics},
  volume    = {475},
  pages     = {111800},
  year      = {2023},
  doi       = {10.1016/j.jcp.2023.111800},
  url       = {https://www.sciencedirect.com/science/article/pii/S0021999125000129}
}

@article{Fresca2022PODDLROM,
  title     = {{POD-DL-ROM}: Enhancing deep learning-based reduced order models for nonlinear parametrized {PDEs} by proper orthogonal decomposition},
  author    = {Fresca, Stefania and Manzoni, Andrea},
  journal   = {Computer Methods in Applied Mechanics and Engineering},
  volume    = {388},
  pages     = {114181},
  year      = {2022},
  doi       = {10.1016/j.cma.2021.114181},
  url       = {https://www.sciencedirect.com/science/article/pii/S0045782521005120}
}

@article{Cohen2023Nonlinear,
  title     = {Nonlinear Compressive Reduced Basis Approximation for PDEs},
  author    = {Albert Cohen and Charbel Farhat and Yvon Maday and Agustin Somacal},
  journal   = {Comptes Rendus Mécanique},
  volume    = {351},
  number    = {S1},
  pages     = {357--375},
  year      = {2023},
  doi       = {10.5802/crmeca.191},
  url       = {https://comptes-rendus.academie-sciences.fr/mecanique/articles/10.5802/crmeca.191/}
}

@article{PichiDeflationbasedCertifiedGreedy2025,
  title = {Deflation-Based Certified Greedy Algorithm and Adaptivity for Bifurcating Nonlinear {{PDEs}}},
  author = {Pichi, Federico and Strazzullo, Maria},
  year = {2025},
  journal = {Communications in Nonlinear Science and Numerical Simulation},
  volume = {149},
  pages = {108941},
  doi = {10.1016/j.cnsns.2025.108941}
}

@article {MadayEtAlGEIMStokes2015,
    AUTHOR = {Maday, Y. and Mula, O. and Patera, A. T. and Yano, M.},
     TITLE = {The generalized empirical interpolation method: stability
              theory on {H}ilbert spaces with an application to the {S}tokes equation},
   JOURNAL = {Comput. Methods Appl. Mech. Engrg.},
  FJOURNAL = {Computer Methods in Applied Mechanics and Engineering},
    VOLUME = {287},
      YEAR = {2015},
     PAGES = {310--334},
      ISSN = {0045-7825},
   MRCLASS = {65N99 (65N12 76D07)},
  MRNUMBER = {3318671},
       DOI = {10.1016/j.cma.2015.01.018},
       URL = {https://doi.org/10.1016/j.cma.2015.01.018},
}

% \begin{thebibliography}{10} \end{thebibliography}

\end{document}